\documentclass[11pt,english,a4paper]{smfart}

\usepackage[francais]{babel}
 \usepackage{vaucanson-g}
\usepackage{graphicx}
\usepackage{ae,amsfonts,euscript,enumerate}
\usepackage[cm]{aeguill}

\renewcommand{\theequation}{\thesection.\arabic{equation}}

\newtheorem{thm}{Theorem}[section]

\newtheorem{defn}{Definition}[section]
\newtheorem{notn}{Assumption--Notation}
\newtheorem{prop}{Proposition}[section]
\newtheorem{lem}{Lemma}[section]
\newtheorem{cor}{Corollary}[section]
\newtheorem{conj}{Conjecture}[section]

\begin{document}
\bibliographystyle{plain}
\title[]{A problem about Mahler functions }
\author{Boris Adamczewski}
\address{
CNRS, Universit\'e de Lyon, Universit\'e Lyon 1\\
Institut Camille Jordan  \\
43 boulevard du 11 novembre 1918 \\
69622 Villeurbanne Cedex, France}

\email{Boris.Adamczewski@math.univ-lyon1.fr}

\author{Jason P.~Bell}
\thanks{The first author was supported by the project  HaMoT, ANR 2010 BLAN-0115-01. The second author was supported by NSERC grant 31-611456.}

\address{
Department of Pure Mathematics\\
University of Waterloo\\
Waterloo, ON, Canada\\
 N2L 3G1}

\email{jpbell@uwaterloo.ca}




\bibliographystyle{plain}

\begin{abstract}  
Let $K$ be a field of characteristic zero and $k$ and $l$ be two multiplicatively independent positive integers. 
We prove the following result that was conjectured by Loxton and van der Poorten during the Eighties: 
a power series $F(z)\in K[[z]]$ satisfies both a $k$- and a $l$-Mahler type functional equation if and only if it is 
a rational function. 
\end{abstract}
\dedicatory{In memory of Alf van der Poorten}

\maketitle

\tableofcontents

\section{Introduction}

In a series of three papers \cite{Mah29,Mah30A,Mah30B} published in 1929 and 1930, Mahler initiated a 
totally new direction in transcendence theory.   {\it Mahler's method}, a term coined much later by  Loxton and van der Poorten,  
aims at proving transcendence and algebraic independence of values at algebraic points of 
locally analytic functions satisfying certain type of functional equations. 
In its original form, it concerns equations of the form   
\begin{equation}\label{eq: origin}
F(z^k) = R(z,F(z)) \, , 
\end{equation}
where $R(z,x)$ denotes a bivariate rational function with coefficients in a number field.   
For instance, using the fact that $F(z)= \sum_{n=0}^{\infty} z^{2^n}$ satisfies 
the basic functional equation
$$
F(z^2) = F(z) - z \, ,
$$
Mahler was able to prove  that $F(\alpha)$ is a transcendental number for every algebraic 
number $\alpha$ with $0<\vert \alpha \vert <1$. 
As observed by Mahler himself, his approach allows one to deal with functions of several variables 
and systems of functional equations as well.  
It also leads to algebraic independence results, transcendence measures, measures of algebraic independence, and so forth.
Mahler's method was later developed by various authors, including Becker, Kubota, Loxton and 
van der Poorten, Masser, Nishioka, T\"opfer, among  others.  
For classical aspects of Mahler's theory, we refer the reader to the monograph of Ku. Nishioka \cite{Ni_Liv} and the reference therein.  
However, a major deficiency of Mahler's method is that, contrary to Siegel E- and G-functions,  
there is not a single classical transcendental constant that is known to be the value at an algebraic point of an analytic function 
solution to a Mahler-type functional equation\footnote{
A remarkable discovery of Denis, which deserves to be better 
understood, is that Mahler's
method can be also applied to prove transcendence and algebraic 
independence results involving
{\em periods of $t$-modules} which are variants of the more classical 
periods
of abelian varieties, in the framework of the arithmetic of function 
fields of
positive characteristic. For a detailed discussion on
this topic, we refer the reader to the recent survey by Pellarin  \cite{Pel2}, see also \cite{Pel1}.}. 
This may explain why it was somewhat neglected for almost fifty years.

\medskip

At the beginning of the Eighties, Mahler's method really took on a new significance after Mend\`es 
France popularized the fact that some Mahler-type systems of functional equations  
naturally arise in the study of automata theory (see for instance \cite{MF}). 
Though already noticed in 1968 by Cobham \cite{Cob68}, this connection remained relatively unknown at that time, 
probably because Cobham's work was never published in an academic journal.  
Cobham claimed that Mahler's method has the following nice consequence for the Hartmanis--Stearns 
problem about the computational complexity of algebraic irrational real numbers \cite{HS}: 
the expansion of an algebraic irrational number in an integer base cannot be generated by a finite automaton. 
His idea was  to derive this result by applying Mahler's method to 
systems of functional equations of the form 
\begin{equation}\label{eq: new}
\left( \begin{array}{ c }
     F_1(z^k) \\
     \vdots \\
     F_n(z^k)
  \end{array} \right) = A(z)\left( \begin{array}{ c }
     F_1(z) \\
     \vdots \\
     F_n(z)
  \end{array} \right) + B(z) \, ,
\end{equation}
where $A(z)$ is an $n\times n$ matrix and $B(z)$ is an $n$-dimensional vector, both having entries that are rational functions 
with algebraic coefficients.    
Though Cobham's conjecture is now proved in \cite{AdBu07} by mean of a completely different approach,  it still 
remains a challenging problem to complete the proof he envisaged.  In this direction, a great deal of work has been done by 
Loxton and van der Poorten \cite{LvdP82,LvdP88}  and 
a particular attention was then paid to systems of functional equations as in  (\ref{eq: new}) (see for instance \cite{Ni90,Ni91,Ni_Liv,Bec}).  

Let $K$ be field. We observe that a power series $F(z)\in K[[z]]$ is a component 
of a vector satisfying a system of functional equations of the form (\ref{eq: new})\footnote{We assume here that the entries of $A(z)$ and $B(z)$ are  in $K(z)$.} 
if and only if the family 
 $$
 1,F(z),F(z^k),F(z^{k^2}),\ldots
 $$ 
 is linearly dependent over the field 
$K(z)$, that is, if there exist a natural number $n$ and polynomials 
$A(z),P_0(z),\ldots , P_n(z)\in K[z]$, not all of which are zero, such that
\begin{equation} \label{eq: mahler}
 A(z) + \sum_{i=0}^n P_i(z) F(z^{k^i}) \ = \ 0.
\end{equation}  
Following Loxton and van der Poorten \cite{LvdP88}, we say that a power series $F(z)\in K[[z]]$ is 
a $k$-\emph{Mahler function}, or for short is $k$-\emph{Mahler},  if it satisfies a functional equation of the form (\ref{eq: mahler}).

\medskip

Beyond transcendence, Mahler's method and automata theory,  
it is worth mentioning that Mahler functions naturally occur as generating functions  in various other topics such as  
 combinatorics of partitions, numeration and the analysis of algorithms (see \cite{DF} and the references therein and 
 also dozens of examples in \cite{AS2,AS3} and \cite[Chapter 14]{EPSW}). A specially intriguing appearance of Mahler functions is 
 related to the study of Siegel $G$-functions and in particular of diagonals of rational functions\footnote{See for instance 
 \cite{AdBe} for a discussion of the links between diagonals of rational functions with algebraic coefficients and $G$-functions.}. 
 Though no general result confirms this claim, 
 one observes that many generating series associated with the $p$-adic 
 valuation of the coefficients of $G$-functions with rational coefficients turn out to be $p$-Mahler functions.  

As a simple illustration, we give the following example. Let us consider the algebraic function 
$$
\mathfrak f(z) := \frac{1}{(1-z)\sqrt{1-4z}} = \sum_{n=0}^{\infty} \sum_{k=0}^n {2k\choose k} z^n \, 
$$
and define the sequence 
$$
a(n) := \nu_3\left(\sum_{k=0}^n {2k\choose k}\right)\, ,
$$ 
where $\nu_3$ denotes the $3$-adic valuation.   
We claim that the function $$\mathfrak f_1(z) :=\sum_{n\geq 0} a(n)z^n \in \mathbb Q[[z]]$$ is a $3$-Mahler function. 
This actually comes from the following nice equality  
\begin{equation}\label{eq: zag}
\nu_3\left(\sum_{k=0}^n {2k\choose k}\right) = \nu_3\left (n^2 {2n\choose n} \right) \, ,
\end{equation}
independently  proved by Allouche and Shallit in 1989 (unpublished) and by Zagier \cite{Zagier}. Indeed,  
setting $\mathfrak f_2(z) := \sum_{n\geq 0} a(3n)z^n$ 
and $\mathfrak f_3(z) := \sum_{n\geq 0} f(3n+1)z^n$, we infer from Equality (\ref{eq: zag})   that 
\begin{equation*}
\left( \begin{array}{ c }
     \mathfrak f_1(z^3) \\ \\
     \mathfrak f_2(z^3) \\ \\
     \mathfrak f_3(z^3)
  \end{array} \right) = 
A(z)\left( \begin{array}{ c }
     \mathfrak f_1(z) \\ \\
    \mathfrak f_2(z) \\ \\
     \mathfrak f_3(z)
  \end{array} \right) \\  \\
   + B(z)  \, ,
\end{equation*}
with 
$$
A(z) :=   \frac{1}{z^3(1+z+z^2)}
  \left(\begin{array}{ccc}
  z(1+z+z^2)&-z^2&-z\\\\
  0&z^2(1+z)&-z^4\\\\
  0&-z^2&z^2(1+z)
  \end{array}\right)
  $$
and 
$$
B(z) := \displaystyle \frac{1}{z^3(1+z+z^2)}
 \left( \begin{array}{ c }
  \displaystyle  \frac{z(2z^2-1)}{z-1} \\ \\
  \displaystyle   -\frac{z^4}{z-1} \\\\
    \displaystyle \frac{z^2(1+z)}{z-1}
  \end{array} \right)  \, .
$$
A simple computation then gives the relation
$$
a_0(z) + a_1(z)\mathfrak f_1(x) + a_2(z)\mathfrak f_1(z^3) + a_3(z)\mathfrak f_1(z^9)+a_4(z)\mathfrak f_1(z^{27}) =0 \, ,
$$
where 
$$\begin{array}{ll}
a_0(z):= & z+2z^2-z^3+z^4+3z^5-z^7+3z^8+z^9-z^{11}+3z^{12}-2z^{14} \\ 
    &-z^{15} +2z^{16}-2z^{17}-2z^{18}+2z^{21},\\
a_1(z):=&-1-z^4-z^8+z^9+z^{13}+z^{17},\\
a_2(z):=& 1+z+z^2+z^3+z^4+z^5+z^6+z^7+z^8-z^{13}-z^{14}-z^{15}-z^{16}\\
   & -z^{17}-z^{18}-z^{19}-z^{20}-z^{21},\\
a_3(z):=&-z^3-z^6-z^7-z^9-z^{10}-z^{11}-z^{13}-z^{14}+z^{16}-z^{17}+z^{19}\\ 
&+z^{20}+z^{22}+z^{23}+z^{24}+z^{26}+z^{27}+z^{30}, \\
a_4(z):=&z^{21}-z^{48} \, .
\end{array}
$$
Of course, considering the Hadamard product (denoted by $\odot$ below) of several algebraic functions 
would lead to similar examples associated 
with transcendental $G$-functions. For instance, 
the elliptic integral 
$$
\mathfrak g(z):= \frac{2}{\pi} \int_{0}^{\pi/2} \frac{d\theta}{\sqrt{1-16z\sin^2 \theta}}
 = \frac{1}{\sqrt{1-4z}} \odot \frac{1}{\sqrt{1-4z}} = \sum_{n=0}^{\infty} {2n\choose n}^2 z^n
$$
is a transcendental $G$-function and it is not hard to see that, for every prime $p$, 
$$
\mathfrak g_p(z) := 
\sum_{n=0}^{\infty} \nu_p\left( {2n\choose n}^2 \right)z^n
$$
is a $p$-Mahler function. 

\medskip

Regarding (\ref{eq: origin}), (\ref{eq: new}) or (\ref{eq: mahler}), 
it is tempting to ask about the significance of the integer 
parameter $k$. Already in 1976, van der Poorten \cite{vdP76} 
suggested that two solutions of Mahler-type functional equations associated with essentially 
distinct parameters should be completely different. 
For instance,  one may naturally expect \cite{vdP76} (and it is now proved \cite{Ni94}) that the two functions 
$$
\sum_{n=0}^{\infty} z^{2^n} \; \mbox{ and } \; \sum_{n=0}^{\infty} z^{3^n}
$$
are algebraically independent over $\mathbb C(z)$.   
This idea was later formalized by Loxton and van der Poorten who made a 
general conjecture whose one-dimensional version can be stated as follows. 

\begin{conj}[Loxton and van der Poorten]\label{conj} 
Let $k$ and $l$ be two multiplicatively independent positive integers and $L$ be a number field. Let $F(z)\in L[[z]]$ 
be a locally analytic function that is both $k$- and $\ell$-Mahler. Then $F(z)$ must be a rational function. 
\end{conj}

We recall that two integers $k$ and $l$  larger than $1$ are 
multiplicatively independent if there is no pair of positive integers $(n,m)$ such that $k^n=\ell^m$, or equivalently, 
if $\log (k)/\log(\ell) \not\in \mathbb Q$. Conjecture \ref{conj} first appeared in print in 1987 in a paper of van der Poorten \cite{vdP87}. 
Since then it was explicitly studied in a number of different contexts including in some papers of Loxton \cite{Lox88}, Becker \cite{Bec}, 
Rand\'e \cite{Ran93}, Bell \cite{Bell} and the monograph of 
Everest {\it et al.}  \cite{EPSW}.  Independently, Zannier also considered a similar question  in \cite{Zannier}. 

\medskip

In this paper, our aim is to prove the following result. 

\begin{thm} Let $K$ be a field of characteristic zero and let $k$ and $l$ be two multiplicatively independent positive integers.  
Then a power series $F(z)\in K[[z]]$ is both $k$- and $\ell$-Mahler if and only if  it is a rational function.
\label{thm: main}
\end{thm}

Let us make few comments on this result.

\medskip

\begin{itemize}

\item[$\bullet$] Taking $K$ to be a number field in Theorem \ref{thm: main} gives Conjecture \ref{conj}.

\medskip

\item[$\bullet$] If $k$ and $\ell$ denote 
 two multiplicatively dependent natural numbers, then a power series 
is $k$-Mahler if and only if it is also $\ell$-Mahler.

 \medskip

\item[$\bullet$] As explained in more details in Section \ref{sec: cobham}, one motivation for proving Theorem \ref{thm: main} 
is that it provides a far-reaching generalization of 
one fundamental result in the theory of sets of integers recognizable by finite automata: Cobham's theorem. 
 Loxton and van der Poorten  \cite{Lox88,vdP87} actually guessed that Conjecture \ref{conj} should be 
a consequence of  some algebraic independence results for Mahler functions of several variables. 
 In particular, they hoped to obtain a totally 
new proof of Cobham's theorem by using Mahler's method.  Note, however, that our proof of Theorem \ref{thm: main}
follows a totally different way and ultimately relies on Cobham's theorem, so we do not obtain an independent derivation of that result.

\medskip

\item[$\bullet$] Another important motivation for establishing Theorem \ref{thm: main} comes from the fact that these kind of statements, 
though highly natural and somewhat ubiquitous, are usually very difficult to prove.   
In particular, similar independence phenomena, involving two multiplicatively independent integers, 
are expected in various contexts but only very few results 
have been obtained up to now.  As an illustration, we quote below three interesting open problems that rest on such a principle,  
all of them being widely open\footnote{In all of these problems, the integers $2$ and $3$ may of course be replaced by any two 
multiplicatively independent 
integers larger than $1$. This list of problems is clearly not exhaustive and could be
easily enlarged.}.  
A long-standing question in dynamical systems is the so-called $\times\, 2 \, \times3$ problem addressed 
by Furstenberg \cite{Fu}:    
prove that  the only Borel measures on $[0,1]$ that are simultaneously ergodic for 
$T_2(x)=2x\pmod 1$ and 
$T_3(x)=3x\pmod 1$ are the Lebesgue measure and measures supported by 
those orbits that are periodic for both actions $T_2$ and $T_3$. 
The following problem, sometimes attributed to 
Mahler, was suggested by Mend\`es France in \cite{MF} 
(see also \cite{AB}): 
given a binary sequence $(a_n)_{n\geq 0}\in \{0,1\}^{\mathbb N}$, prove that 
$$
\sum_{n= 0}^{\infty} \frac{a_n}{2^n} \; \mbox{ and } \;\sum_{n= 0}^{\infty} \frac{a_n}{3^n}
$$
are both algebraic numbers only if both are rational numbers.  
The third problem we mention 
appeared implicitly in work of Ramanujan (see \cite{Waldschmidt}): 
prove that  both $2^x$ and $3^x$ are integers only if $x$ is a natural number.  
This is a particular instance of the four exponentials conjecture, 
a famous open problem in transcendence theory \cite[Chapter 1, p. 15]{Wald}.

\end{itemize}

\medskip

The outline of the paper is as follows. In Section \ref{sec: cobham}, we briefly discuss the connection 
between Theorem \ref{thm: main} and Cobham's theorem. In Section \ref{sec: strategy}, we describe our strategy for proving 
 Theorem \ref{thm: main}. Then the remaining Sections 4--11 are devoted to the different steps of the proof of Theorem \ref{thm: main}.  



\section{Connection with finite automata and Cobham's theorem}\label{sec: cobham}

One motivation for proving Theorem \ref{thm: main} is that it provides a far-reaching generalization of 
one fundamental result in the theory of sets of integers recognizable by finite automata. 
The aim of this section is to briefly describe this connection. For more details on automatic sets and 
automatic sequences, we refer the reader to the book of Allouche and Shallit \cite{AS}.

\medskip

Let $k\ge 2$ be a natural number.  A set ${\mathcal N}\subset \mathbb N$ 
is said to be $k$-\emph{automatic} if there is a finite-state machine that accepts 
as input the expansion of $n$ in base $k$ and outputs $1$ if $n\in {\mathcal N}$ and $0$ otherwise.  
For example, the set of Thue--Morse integers 
$1,2,4,7,8,11,13,\ldots$, formed by the integers whose sum of binary 
digits is odd, is $2$-automatic. 
The associated automaton is given in Figure 1 below. It has two states.  
This automaton successively reads the binary digits of $n$ 
(starting, say, from the most significant digit and the initial state $q_0$) and thus 
ends the reading either in state $q_0$ or in state $q_1$.   The initial 
state $q_0$ gives the output $0$, while $q_1$ gives the output $1$. 

\begin{figure}[htbp]
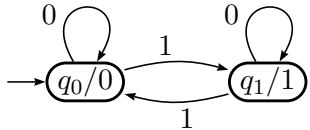

\centering
\VCDraw{%
        \begin{VCPicture}{(0,-1)(4,2)}
 \StateVar[q_0/0]{(0,0)}{a}  \StateVar[q_1/1]{(4,0)}{b}
\Initial[w]{a}
\LoopN{a}{0}
\LoopN{b}{0}
\ArcL{a}{b}{1}
\ArcL{b}{a}{1}
\end{VCPicture}
}
\caption{The finite-state automaton recognizing the set of Thue--Morse integers.}
  \label{AB:figure:thue}
\end{figure}

Another typical $2$-automatic set of integers is given by the powers of $2$: $1,2,4,8,16,\ldots$. 
Though these integers have very simple expansions in base $2$, one can observe that 
this is not the case when writing them in base $3$. 
One of the most important results in the theory of automatic sets formalizes this idea. 
It says that only very well-behaved sets of integers can be automatic with respect to two 
multiplicatively independent numbers. Indeed, in 1969 Cobham \cite{Cobham} 
 proved the following result.

\begin{thm}[Cobham] Let $k$ and $\ell$ be two multiplicatively independent integers. 
Then a set ${\mathcal N} \subseteq \mathbb N$ is both $k$- and $\ell$-automatic if and only if 
it is the union of a finite set and a finite number of arithmetic progressions.  
\label{thm: cob}
\end{thm}

The proof given by Cobham of his theorem is elementary but notoriously difficult and it remains a challenging problem 
to find a more natural/conceptual proof (see for instance the comment in Eilenberg \cite[p. 118]{Eil}). 
There are many interesting generalizations of this result.  A very recent one is due to 
Durand \cite{Dur} and  
we refer the reader to the introduction of \cite{Dur} for a brief but complete discussion 
about such generalizations. 

\medskip

To end this section, let us briefly explain why Cobham's Theorem is a consequence of Theorem \ref{thm: main}. 
Let us assume that ${\mathcal N} \subseteq \mathbb N$ is both $k$- and $\ell$-automatic for multiplicatively independent 
natural numbers $k$ and $\ell$. Set $F(x) := \sum_{n\in \mathcal N} x^n\in \mathbb Z[[x]]$. Then it is known that $F(x)$ is both $k$- and $\ell$-Mahler 
(see for instance \cite[p. 232]{EPSW}). 
By Theorem \ref{thm: main}, it follows that $F(x)$ is a rational function and thus the sequence of coefficients of $F(x)$ does satisfy a linear recurrence. 
Since the coefficients of $F(x)$ take only a two distinct values ($0$ and $1$), we see that this linear recurrence is ultimately periodic. This 
exactly means that $\mathcal N$ is the union of a finite set and a finite number of arithmetic progressions, as claimed by Cobham's theorem.



\section{Sketch of proof of Theorem \ref{thm: main}}\label{sec: strategy}

In this section, we describe the main steps of the proof 
of Theorem \ref{thm: main}. 

\medskip

Let $R$ be a ring and $\mathfrak P$ be an ideal of $R$. If $F(x) = \sum_{n=0}^{\infty} f(n)x^n\in R[[x]]$, 
then we denote by $F_{\mathfrak P}(x)$ the 
reduction of $F(x)$ modulo $\mathfrak P$, that is 
$$
F_{\mathfrak P}(x) = \sum_{n=0}^{\infty} (f(n)\bmod \mathfrak P) x^n\in (R/{\mathfrak P})[[x]] \, .
$$
Let $K$ be a field of characteristic zero and $F(x)\in K[[x]]$ be both $k$- and $\ell$-Mahler.  

\medskip
 
\noindent{\bf Step 0.} This is a preliminary step. 
In the introduction, we defined Mahler functions as those satisfying Equation (\ref{eq: mahler}) 
but it is not always convenient to work with this general form of equations.  
In Sections \ref{sec: form1} and \ref{sec: form2} we show  
that there is no loss of generality  to work with some more restricted types of functional equations. 
Also in Section \ref{sec: kl}, we prove that one can assume 
without loss of generality some additional assumptions on $k$ and $\ell$; namely that there are 
primes $p$ and $q$ such that  $p$ divides $k$ but does not divide $\ell$ and $q$ divides $\ell$ but 
does not divide $k$. 

\medskip
 
\noindent{\bf Step 1.} 
A first observation, proved in Section \ref{sec: nf}, is that the coefficients of the formal power series $F(x)$ 
only belong to some finitely generated $\mathbb Z$-algebra $R\subseteq K$.  Then we prove the following 
useful local--global principle: 
$F(x)$ is a rational function if  it has rational reduction modulo a sufficiently large set of maximal ideals of $R$. 
Using classical results of commutative algebra about Jacobson rings,  
we derive from our local--global principal that there is no loss of generality to assume that $K$ is a number field and that $R$ 
is the principal localization of a number ring.  

\medskip

\noindent{\bf Comment.} Our strategy consists now in applying again our local--global principle.  
Indeed, since $R$ is  the principal localization of a number ring, we have that the quotient ring $R/\mathfrak P$ is a finite field 
for every prime ideal $\mathfrak P$ of $R$.  Our plan is thus to take advantage of the fact that $F_{\mathfrak P}(x)$ has
coefficients in the finite set $R/\mathfrak P$ to prove that $F_{\mathfrak P}(x)$ is both 
a $k$- and an $\ell$-automatic power series, for some prime ideals $\mathfrak P$. 
If this is the case, then Cobham's theorem applies and we get that $F_{\mathfrak P}(x)$ is a rational function.  
The local--global principle actually implies that it is enough to prove that $F_{\mathfrak P}(x)$ is both 
$k$- and $\ell$-automatic  for infinitely many prime ideals $\mathfrak P$ of $R$.

\medskip

\noindent{\bf Step 2.} In Section \ref{sec: auto}, we underline 
the relation between $k$-Mahler, $k$-regular, and $k$-automatic power series.
In particular, we show that every $k$-Mahler power series can 
be decomposed as 
$$
F(x) = G(x)\cdot  \Pi(x) \,,
$$
where $G(x)\in R[[x]]$ is a $k$-regular power series  and 
$\Pi(x)\in R[[x]]$ is the inverse of an infinite product of polynomials.  
Since $F(x)$ is also $\ell$-Mahler, we also have a similar decomposition 
$$
F(x) = H(x)\cdot  \Pi'(x) \,,
$$
where $H(x)\in R[[x]]$ is a $\ell$-regular power series  and 
$\Pi'(x)\in R[[x]]$ is the inverse of an infinite product of polynomials.  
Furthermore, the theory of regular power series 
implies that $G_{\mathfrak P}(x)$ is $k$-automatic and that $H_{\mathfrak P}(x)$ is 
$\ell$-automatic for every prime ideal $\mathfrak P$ of $R$. 

\medskip

In Section \ref{sec: main} we will split both infinite products $\Pi(x)$ and $\Pi'(x)$ and get an expression 
of the form  
$$
F(x) = G(x) \cdot \Pi_1(x) \cdot \Pi_2(x) = H(x) \cdot \Pi'_1(x) \cdot \Pi'_2(x) \, 
$$
where $\Pi_1(x), \Pi_2(x),\Pi'_1(x),\Pi'_2(x)\in R[[x]]$ are inverses of some other infinite products of polynomials.

 \medskip

\noindent{\bf Step 3.} 
In Section \ref{sec: elim}, we look at the singularities of Mahler functions at roots of unity.  
We use asymptotic techniques to show that one can reduce to the case of considering Mahler equations 
whose singularities at roots of unity have a restricted form.  
This ensures, using some results of Section \ref{sec: auto}, that 
$\Pi_1(x)$ is $k$-automatic and that $\Pi'_1(x)$ is $\ell$-automatic 
when reduced modulo every prime ideal $\mathfrak P$ of $R$.

\medskip

\noindent{\bf Step 4.} 
In our last step, we use Chebotarev's density theorem in order to ensure the existence of an 
infinite set ${\mathcal S}$ of  prime ideals of $R$ 
such that  $\Pi_2(x)$ is $k$-automatic  and $\Pi'_2(x)$ is  $\ell$-automatic
when reduced modulo every ideal $\mathfrak P\in \mathcal S$. 

\medskip

\noindent{\bf Conclusion.} 
Since the product of $k$-automatic power series is $k$-automatic, we infer from Steps 2, 3 and 4  
that for every prime ideals $\mathfrak P\in\mathcal S$
the power series $F_{\mathfrak P}(x)$ is both $k$- and $\ell$-automatic.  
By Cobham's theorem, $F_{\mathfrak P}(x)$ is rational for every such prime ideal.  
Then the local--global principle ensures that $F(x)$ is rational, as desired. 



\section{Preliminary reduction for the form of Mahler equations}\label{sec: form1} 

In the introduction, we define $k$-Mahler functions as power series satisfying a functional equation 
of the form given in (\ref{eq: mahler}).  In the literature, they are sometimes defined as solutions of a more restricted 
type of functional equations. We recall here that these apparently stronger conditions on the functional equations actually lead 
to the same class of functions. In the sequel, it will thus be possible to work without loss of generality with these more restricted type of equations. 

\begin{lem}\label{lem: reduction1}
Let us assume that $F(x)$ satisfies a $k$-Mahler equation as in (\ref{eq: mahler}). 
Then there exist polynomials $P_0(x),\ldots,P_n(x)$ in $K[x]$, 
with $\gcd(P_0(x),\ldots,P_n(x))=1$ and $P_0(x)P_n(x)\not=0$, and such that 
\begin{equation}
\label{eq: mahler2}
\sum_{i=0}^nP_i(x)F(x^{k^i})  = 0 \, .
\end{equation}
\end{lem}

\begin{proof}
Let us assume that $F(x)$ satisfies a $k$-Mahler  equation as in (\ref{eq: mahler}). 
There thus exist some nonnegative integer $n$ and polynomials $A(x),A_0(x),\ldots,A_n(x)$ in $K[x]$, with $A_{n}(x)$ nonzero, 
such that 
$$
\sum_{i=0}^{n} A_i(x)F(x^{k^i}) = A(x) \, .
$$

We first show that we can assume that $A(x)=0$. Indeed, let us assume that $A(x)\not=0$. 
Applying the operator $x\mapsto x^k$ to this equation, we get that 
$$
\sum_{i=0}^{n} A_i(x^k)F(x^{k^{i+1}}) = A(x^k) \, .
$$
Multiplying the first equation by $A(x^k)$ and the second by $A(x)$ and subtracting, we obtain the new equation
$$ 
\sum_{i=0}^{n+1} B_i(x)F(x^{k^i}) = 0 \, ,
$$
where $B_i(x) := A_i(x)A(x^k)-A_i(x^x)A(x)$ for every integer $i$, $1\leq i\leq n$ and where 
$B_{n+1}:= A_n(x^k)A(x)\not=0$. We can thus assume without loss of generality that $A(x)=0$.  

\medskip

Now among all such nontrivial relations of the form 
\begin{equation}\label{eq: mahlerbis}
\sum_{i=0}^{n} P_i(x)F(x^{k^i}) = 0 \, ,
\end{equation}
we choose one with $n$ minimal.  Thus $P_n(x)$ is nonzero. 
We claim $P_0(x)$ is nonzero.  Let us assume this is not the case. 
Pick the smallest integer $j$ such that $P_j(x)$ is nonzero. By assumption, $j >0$. 
Then there is some nonnegative integer $a$ such that the coefficient of 
$x^a$ in $P_j(x)$ is nonzero.  Let $b$ be the unique integer such that $a\equiv b \bmod k$ and 
$0\leq b < k$.  Let us define the operator $\Lambda_b$ from $K[[x]]$ into itself by 
$$
\Lambda_b\left(\sum_{i=0}^{\infty} f(i)x^{i}\right) := \sum_{i=0}^{\infty} f(ki + b)x^{i} \, .
$$
Then every $F(x)\in K[[x]]$ has a unique decomposition as 
$$
F(x) = \sum_{b=0}^{k-1} x^{b} \Lambda_b(F)(x^k) \, ,
$$
which implies that 
$$
\Lambda_b\left(F(x)G(x^k)\right) = \Lambda_b\left(F(x)\right)G(x)
$$
for every pair of power series $F(x),G(x)\in K[[x]]$. 
Applying $\Lambda_b$ to Equation (\ref{eq: mahlerbis}), 
we thus get that 
$$
0= \Lambda_b\left( \sum_{i=j}^{n} P_i(x)F(x^{k^i})\right) = 
\sum_{i=j-1}^{n-1}\Lambda_b\left(P_{i+1}(x)\right) F(x^{k^i}) \, .
$$
By construction, $\Lambda_b(P_j(x))$ is nonzero, which shows that this relation is nontrivial. 
This contradicts the minimality of $n$.  It follows that $P_0(x)$ is nonzero.

\medskip

Furthermore, if $\gcd(P_0(x),\ldots,P_n(x))=D(x)\not=0$, it suffices to divide (\ref{eq: mahlerbis}) 
by $D(x)$ to obtain an equation with the desired properties. This  ends the proof. 
\end{proof}



\section{Reduction to the number field case}\label{sec: nf}

In this section, we show that we may restrict our attention to the case where the base field $K$ 
is replaced by a number field and more precisely by a principal localization of a number ring.

\begin{thm}  
Let us assume that the conclusion of Theorem \ref{thm: main} holds whenever the field $K$ is replaced by 
a principal localization of a number ring. Then Theorem \ref{thm: main} is true. 
\label{thm: red1}
\end{thm}

We first observe that the coefficients of a Mahler function in $K[[x]]$ actually belong to 
some finitely generated $\mathbb{Z}$-algebra $R\subseteq K$. 

\begin{lem} Let $K$ be a field of characteristic zero, let $k\geq 2$ be an integer, and let 
$F(x)\in K[[x]]$ be a $k$-Mahler power series.  Then there exists a finitely generated 
$\mathbb{Z}$-algebra $R\subseteq K$ such that $F(x)\in R[[x]]$.
\label{lem: KtoR}
\end{lem}

\begin{proof}  
We first infer from Lemma \ref{lem: reduction1} that there exist a natural number $n$ and 
polynomials $P_0(x),\ldots ,P_n(x)\in K[x]$ with $P_0(x)P_n(x)\neq 0$ such that 
$$
\sum_{i=0}^n P_i(x)F(x^{k^i}) \ = \ 0 \,.
$$ 

Let $d$ be a natural number that is strictly greater than the degrees of the polynomials $P_0(x),\ldots ,P_n(x)$. 
Let $R$ denote the smallest $\mathbb{Z}$-algebra countaining:
\begin{itemize}

\medskip
 
\item the coefficients of $P_0(x),\ldots ,P_n(x)$; 

\medskip

\item the coefficients $f(0),\ldots ,f(d)$;

\medskip

\item the multiplicative inverses of all nonzero coefficients of $P_0(x)$\,. 

\end{itemize}

\medskip

By definition, $R\subseteq K$ is a finitely generated $\mathbb{Z}$-algebra. 
We claim that $F(x)\in R[[x]]$.  To see this, suppose that this is not the case.  
Then there is some smallest natural number $n_0$ such that $f(n_0)\not\in R$.   
Furthermore, $n_0 > d$.  
Consider the equation
\begin{equation}\label{eq: PF} 
P_0(x)F(x) = -\sum_{i=1}^n P_i(x)F(x^{k^i}) \,.
\end{equation}
Let $i$ denote the order of $P_0(x)$ at $x=0$ and 
let $c\not =0$ denote the coefficient of $x^i$ in $P_0(x)$.  
Then if we extract the coefficient of $x^{n_0+i}$ in Equation (\ref{eq: PF}), we see that
$cf(n_0)$ can be expressed as an $R$-linear combination of $f(0),\ldots,f(n_0-1)$. 
Hence $cf(n_0)$ belongs to $R$ by the minimality of $n_0$.  
Since $c^{-1}\in R$ we see that $f(n_0)\in R$, a contradiction.   
This ends the proof.
\end{proof}

We now prove that the height of a rational function which satisfies a Mahler-type equation 
can be bounded by the maximal of the degrees of the polynomials defining the underlying equation.  

\begin{lem} Let $K$ be a field, let $n$ and $d$ be natural numbers, and let $P_0(x),\ldots ,P_n(x)$ 
be polynomials in $K[x]$ of degree at most $d$ with $P_0(x)P_n(x)\neq 0$.  
Suppose that $F(x)\in K[[x]]$ satisfies  
the Mahler-type equation
$$
\sum_{i=0}^n P_i(x)F(x^{k^i})\ = \ 0 \, .
$$  
If $F(x)$ is rational, then there exist polynomials $A(x)$ and $B(x)$ of degree at most  $d$ with 
$B(0)=1$ such that $F(x)$ is the power series expansion of $A(x)/B(x)$.
\label{lem: AB}
\end{lem}

\begin{proof} Without any loss of generality we can assume that $F(x)$ is not identically zero.  
If $F(x)$ is rational, then there exist two polynomials $A(x)$ and $B(x)$ in $K[x]$ with gcd $1$ and 
with $B(0)=1$ such that $F(x)=A(x)/B(x)$.   Observe that 
$$
\sum_{i=0}^n P_i(x)A(x^{k^i})/B(x^{k^i}) \ = \ 0 \,.
$$
Multiplying both sides of this equation by the product $B(x)B(x^k)\cdots B(x^{k^n})$, 
we see that $B(x^{k^n})$ divides
$$
P_n(x)A(x^{k^n})B(x)\cdots B(x^{k^{n-1}}) \,.
$$
Since $\gcd(A(x),B(x))=1$ and $A(x)$ is nonzero, we actually have that 
$B(x^{k^n})$ divides 
$$
P_n(x)B(x)\cdots B(x^{k^{n-1}}) \,.
$$  
Let $d_0$ denote the degree of $B(x)$.  Then 
we have
\begin{eqnarray*}
k^nd_0 & \le &  {\rm deg}(P_n(x))+\sum_{i=0}^{n-1} {\rm deg}(B(x^{k^i})) \\ \\
&\le & d+d_0(1+k+\cdots +k^{n-1})\\ \\
&=& d + d_0(k^n-1)/(k-1) \,.
\end{eqnarray*}
Thus 
$$
d_0(k^{n+1}-2k^n+1)/(k-1) \le d \,,
$$  
which implies $d_0 \leq d$ since $(k^{n+1}-2k^n+1)/(k-1) \geq 1$ for every integer $k\geq 2$. 
A symmetric argument gives the same upper bound for the degree of $A(x)$. 
\end{proof}

We derive from Lemma \ref{lem: AB} a useful local--global principle for the rationality of Mahler functions with coefficients 
in a finitely generated $\mathbb Z$-algebra. 

\begin{lem} Let $K$ be a field, let $k\geq 2$ be an integer, and let 
$R\subseteq K$ be a finitely generated $\mathbb Z$-algebra.  Let us assume that $F(x)\in R[[x]]$ has the following properties.
\begin{enumerate}
\medskip

\item[{\rm (i)}] There exist a natural number $d$ and polynomials $P_0(x),\ldots ,P_n(x)\in R[x]$ with 
$P_0(x)P_n(x)\neq 0$ such that 
$$
\sum_{i=0}^n P_i(x)F(x^{k^i}) \ = \ 0 \, .
$$

\medskip

\item[{\rm (ii)}] There exists an infinite set $\mathcal{S}$ of maximal ideals of $R$ such that $F(x) \bmod I$  
is a rational power series in $(R/I)[[x]]$ for every $I\in{\mathcal S}$.  

\medskip

\item[{\rm (iii)}] One has $\displaystyle\bigcap_{I\in {\mathcal S}} I = \left\{0\right\}$ .

\end{enumerate}

\medskip

\noindent Then $F(x)$ is a rational function.
\label{lem: rational}
\end{lem}

\begin{proof}
Let $d$ be a natural number that is strictly greater than the degrees of all polynomials $P_0(x),\ldots ,P_n(x)$.
By (ii), we have that for each maximal ideal $I$ in ${\mathcal S}$, $F(x)\bmod I$ is a rational function.  
Thus by (i) and Lemma \ref{lem: AB}, we see that for each maximal ideal $I$ in $\mathcal{S}$, 
there exist two polynomials $A_I(x)$ and $B_I(x)\in \left(R/I\right)[x]$ of degree 
at most $d$ with $B_I(0)=1$ and 
such that $F(x)\equiv A_I(x)/B_I(x)~\bmod ~I$. 
In particular, if  $F(x) = \sum_{j\geq 0} f(j)x^j$,  
we see that the sequences in the set 
$\left\{ (f(d+1+i+j)~\bmod~I)_{j\ge 0}~\mid~i=0,\ldots ,d\right\}$ are linearly dependent over $R/I$.  
Thus the determinant of each $(d+1)\times (d+1)$ submatrix of the infinite matrix 
\[M := \left(\begin{array}{cccc} f(d+1)& f(d+2) & f(d+3) & \cdots \\
f(d+2) & f(d+3) & f(d+4) & \cdots \\
\vdots & \vdots & \vdots & \cdots \\
 f(2d+1) & f(2d+2) & f(2d+3) & \cdots  
 \end{array}
\right)
\] lies in the maximal ideal $I$.  Since this holds for every maximal ideal $I$ in $\mathcal{S}$, 
we infer from (iii) that every $(d+1)\times (d+1)$ minor of $M$ vanishes.  
It follows that $M$ has rank at most $d$ and thus the rows of $M$ are linearly dependent over the field of fractions of 
$R$.  In particular, there exist $c_0,\ldots ,c_d\in R$, not all zero, such that
$$
\sum_{i=0}^d c_{i} f(d+1+i+j)\ = \ 0
$$ 
for all $j\ge 0$.  Letting $B(x) :=c_d+c_{d-1}x+\cdots +c_0x^d$, we see that 
$B(x)F(x)$ is a polynomial. Hence $F(x)$ is a rational function. This ends the proof. 
\end{proof}

We are now ready to prove the main result of this section. 

\begin{proof}[Proof of Theorem \ref{thm: red1}]  
Let $K$ be a field of characteristic zero and let $F(x)\in K[[x]]$ be a power series that is both $k$- and 
$\ell$-Mahler for some multiplicatively independent natural numbers $k$ and $\ell$.  
By Lemma \ref{lem: reduction1}, there are natural numbers $n$ and $m$ and 
polynomials $P_0(x),\ldots ,P_n(x)$ and $Q_0(x),\ldots ,Q_m(x)$ with 
$P_0(x)P_n(x)Q_0(x)Q_m(x)\neq 0$ and such that
\begin{equation}\label{eq: nonvan}
\sum_{i=0}^n P_i(x)F(x^{k^i}) \ =  \ \sum_{j=0}^m Q_j(x)F(x^{\ell^j}) \ =  \ 0 \, .
\end{equation}
Then by Lemma \ref{lem: KtoR}, there is a finitely generated $\mathbb{Z}$-algebra $R\subseteq K$ such that $F(x)\in R[[x]]$.  
By adding all the coefficients of $P_0(x),\ldots ,P_n(x)$ and of $Q_0(x),\ldots ,Q_m(x)$ to $R$, we can assume that 
$P_i(x)$ and $Q_j(x)$ are in $R[x]$ for $(i,j)\in \{1,\ldots , n\}\times \{1,\ldots ,m\}$.  
By inverting the nonzero integers in $R$, we can assume that $R$ is a finitely generated $\mathbb{Q}$-algebra. 

\medskip

Let $\mathcal{M}\subseteq {\rm Spec}(R)$ denote the collection of maximal ideals of $R$.  Since 
$R$ is a finitely generated $\mathbb{Q}$-algebra, $R$ is a Jacobson ring and $R/I$ is a finite 
extension of $\mathbb{Q}$ for every $I\in \mathcal{M}$ (see \cite[Theorem 4.19, p. 132]{Ei}).
Thus, for each maximal ideal $I$ of $R$, the quotient field $R/I$ is a number field.  
If we assume that the conclusion of Theorem \ref{thm: main} holds when the base field is a number field, 
then we get that $F(x)\bmod I$ is a rational function in $(R/I)[[x]]$ for it is clearly 
both $k$- and $\ell$-Mahler\footnote{Note that since $P_0(0)Q_0(0)\not=0$, 
we may assume that $P_0(0)=Q_0(0)=1$ by multiplying the left-hand side of (\ref{eq: nonvan}) by 
$1/P_0(0)$ and the right-hand side of (\ref{eq: nonvan}) by $1/Q_0(0)$. This ensures that, for each functional equation, not all the 
coefficients vanish when 
reduced modulo a maximal ideal $I$ of $R$. Hence $F(x)\bmod I$ is both $k$- and $\ell$-Mahler.}. 
Since $R$ is a Jacobson ring that is also a domain, we have that 
$\bigcap_{I\in{\mathcal M}}I=\{0\}$ (c.f. \cite[p. 132]{Ei}).   
Then Lemma \ref{lem: rational} implies that $F(x)$ is a rational function in $R[[x]]$.  
This shows it is sufficient to prove Theorem \ref{thm: main} in the case that $K$ is a number field.

\medskip

We can thus assume that $F(x)\in K[[x]]$ where $K$ is a number field. 
Now, if we apply again Lemma \ref{lem: KtoR}, we see that 
there is a finitely generated $\mathbb{Z}$-algebra $R\subseteq K$ such that $F(x)\in R[[x]]$. 
Furthermore, every finitely generated $\mathbb{Z}$-subalgebra of a number field $K$ has a generating set of the form 
$\{a_1/b,\ldots ,a_t/b\}$, where $b$ is a nonzero (rational) integer and $a_1,\ldots ,a_t$ are algebraic integers in $K$.  
Thus $R$ is a subalgebra of a principal localization of a number ring, that is $R\subseteq \left(\mathcal{O}_K\right)_b,$ 
where $\mathcal{O}_K$ denotes the ring of algebraic integers in $K$.  Thus to establish Theorem \ref{thm: main} 
it is sufficient to prove the following result: let $k$ and $\ell$ be two multiplicatively independent natural numbers, 
let $R$ be a principal localization of a number ring, and let $F(x)\in R[[x]]$, then if $F(x)$ is both $k$- and $\ell$-Mahler  
it is a rational function. 
This concludes the proof.  \end{proof}



\section{Further reductions for the form of Mahler equations}\label{sec: form2}

In this section, we refine the results of Section \ref{sec: form1}.  
We show that a power series satisfying a Mahler equation of the form given in (\ref{eq: mahler2}) 
is also solution of a more restricted type of functional equations. 

\begin{lem} 
Let $K$ be a field and $k\ge 2$ be an integer. Let us assume that  
$F(x):=\sum_{i\geq 0}f(i)x^{i}\in K[[x]]$ satisifes a $k$-Mahler equation of the form
$$
\sum_{i=0}^n P_i(x)F(x^{k^i}) \  = \ 0 \, ,
$$   
where $P_0(x),\ldots,P_n(x)\in K[x]$, 
 $\gcd(P_0(x),\ldots,P_n(x))=1$ and $P_0(x)P_n(x)\not=0$. 
Then there exists a natural number $N$ such that, for every integer $a>N$ with $f(a)\not = 0$, 
$F(x)$  can be decomposed as 
$$
F(x) = T_a(x) + x^aF_0(x) \, ,
$$ 
where $T_a(x)\in K[x]$  and $F_0(x)$ has nonzero constant term and 
satisfies a $k$-Mahler equation 
$$
\sum_{i=0}^m Q_i(x)F_0(x^{k^i}) \  = \ 0
$$ 
for some natural number $m$ and polynomials $Q_0,\ldots ,Q_m\in K[x]$ satisfying the following conditions. 

\medskip

\begin{enumerate}

\item[{\rm (i)}] One has $Q_0(0)=1$.

\medskip

\item[{\rm (ii)}]  If $\alpha\neq 0$ and $P_0(\alpha)=0$, then $Q_0(\alpha)=0$.

\medskip

\item[{\rm (iii)}]  If $\alpha\neq 0$, $P_0(\alpha)=0$ and $\alpha^k=\alpha$, then $Q_j(\alpha)\neq 0$ for some $j\in\{1,\ldots,m\}$.
\end{enumerate}
\label{lem: gettingridofzeros}
\end{lem}

\begin{proof}  
By assumption, we have that $F(x)$ satisfies a $k$-Mahler equation
$$
\sum_{i=0}^n P_i(x)F(x^{k^i}) \ = \ 0 \, ,
$$  
where $P_0(x)P_n(x)$ is nonzero. 
Let $N$ denote the order of vanishing of $P_0(x)$ at $x=0$.  
Suppose that $a\ge N$ and $f(a)\neq 0$. Then we have that 
$$
F(x)=T_a(x)+x^aF_0(x) \, ,
$$
for some polynomial $T_a(x)$ of degree $a-1$ and 
some power series $F_0(x)$ with nonzero constant term.  
Then we have
$$
\sum_{i=0}^n P_i(x) (T_a(x^{k^i})+x^{k^i \cdot a}F_0(x^{k^i})) \ = \ 0 \, ,
$$
which we can write as
\begin{equation} 
\label{eq: PGT} 
\sum_{i=0}^n P_i(x) x^{k^i \cdot a}F_0(x^{k^i}) \ = \ C(x) \, ,
\end{equation} 
where $C(x)$ denotes the polynomial
$$
C(x) :=-\sum_{i=0}^n  P_i(x) T_a(x^{k^i}) \, .
$$
Set $S(x) := P_0(x)x^{-N}$. By definition of $N$,  $S(x)$ is a polynomial with $S(0)\neq 0$.
Then if we divide both sides of Equation (\ref{eq: PGT}) by 
$x^{a+N}$, we obtain that
\begin{equation}
\label{eq: C0}
S(x)F_0(x)+\sum_{i=1}^n P_i(x) x^{k^i a-a-N}F_0(x^{k^i}) \ = \ x^{-a-N}C(x) \, .
\end{equation}
Observe that the left-hand side is a power series with constant term 
$S(0)F_0(0)\neq 0$ and thus 
$C_0(x) := x^{-a-N}C(x)$ is a polynomial with $C_0(0)\neq 0$. 
Applying the operator $x \mapsto x^k$, we also obtain that
\begin{equation}
\label{eq: C02}
S(x^k)F_0(x^k)+\sum_{i=1}^n P_i(x^k) x^{k^{i+1} a-ka-kN}F_0(x^{k^{i+1}}) \ = \ C_0(x^k) \, .
\end{equation}
Multiplying (\ref{eq: C0}) by $C_0(x^k)$ and (\ref{eq: C02}) by $C_0(x)$ and then subtracting, we get that 
\begin{eqnarray*}
&~& C_0(x^k)S(x)F_0(x)+\sum_{i=1}^n C_0(x^k)P_i(x) x^{k^i a-a-N}F_0(x^{k^i}) \\
&~&
- C_0(x)S(x^k)F_0(x^k) - \sum_{i=1}^n C_0(x)P_i(x^k)x^{k^{i+1} a-ka-kN}F_0(x^{k^{i+1}}) \ = \ 0 \, .
\end{eqnarray*}
Since $C_0(0)$ and $S(0)$ are nonzero, we see that $F_0(x)$ satisfies a non-trivial 
$k$-Mahler equation 
$$
\sum_{i=0}^{n+1} Q_i(x)F_0(x^{k^i}) \ = \ 0 \, ,
$$ 
where
$$
Q_0(x):=\frac{C_0(x^k)S(x)}{\gcd(C_0(x),C_0(x^k))} \, 
$$
and 
$$
Q_1(x) :=\frac{C_0(x^k)P_1(x) x^{k^i a-a-N} - C_0(x)S(x^k)}{\gcd(C_0(x),C_0(x^k))} \, ,
$$ 
and, for $i\in \{2,\ldots ,n+1\}$,
$$
Q_i(x) := \frac{x^{k^i a-ka-N}(C_0(x^k)x^{(k-1)a}P_i(x)-C_0(x)P_{i-1}(x^k))}{\gcd(C_0(x),C_0(x^k))} \, ,
$$ 
with the convention that $P_{n+1}(x):=0$. 
By construction, $Q_0(0)\not = 0$, which we may assume to be equal to $1$ by multiplying our equation by $1/Q_0(0)$.   
Since $S(x)$ divides $Q_0(x)$, we have that if $P_0(\alpha)=0$ for some nonzero $\alpha$ then $Q_0(\alpha)=0$.  
Finally, suppose that $P_0(\alpha)=0$ for some nonzero $\alpha$ such that $\alpha^k=\alpha$.   
We claim that $Q_j(\alpha)$ is nonzero for some $j\in\{1,\ldots,n+1\}$.   
Note that since $\gcd(P_0(x),\ldots,P_n(x))=1$, there is some smallest positive integer $i$ such that $P_i(\alpha)$ is nonzero.  
We claim that $Q_i(\alpha)\not = 0$. Indeed, otherwise 
$\alpha$ would be a root of $C_0(x)/\gcd(C_0(x),C_0(x^k))$, but this is impossible since
$\alpha^k=\alpha$.  This ends the proof.
\end{proof}

\begin{cor} Let $K$ be a field and let $k$ and $\ell$ be multiplicatively independent natural numbers. 
Let $F(x) \in K[[x]]$ be a power series that is both $k$- and $\ell$-Mahler and that is not a polynomial.   
Then there is a natural number $a$ such that $F(x)$ can be decomposed as 
$$
F(x)=T_a(x)+x^aF_0(x) \,,
$$ 
where $T_a(x)$ is a polynomial of degree $a-1$, $F_0(x)$ satisfies a $k$-Mahler equation as in Lemma 
\ref{lem: gettingridofzeros}, and $F_0(x)$ also satisfies an $\ell$-Mahler equation 
of the form 
$$
\sum_{i=0}^r R_i(x)F_0(x^{\ell^i}) \  = \ 0
$$
with $R_0(x),\ldots,R_r(x)\in K[x]$ and $R_0(0)=1$.
\label{cor: P01Q01}
\end{cor}

\begin{proof} 
The result follows directly by applying Lemma \ref{lem: gettingridofzeros} twice to $F(x)$, viewed respectively 
as a $k$-Mahler and an $\ell$-Mahler function, and then 
by choosing $a$ large enough.
\end{proof}



\section{Links with automatic and regular power series}\label{sec: auto}

The aim of this section is to underline 
the relation between $k$-Mahler, $k$-regular, and $k$-automatic power series.
We gather some useful facts about automatic and regular power series that will turn out to be useful 
for proving Theorem \ref{thm: main}.  We also observe that every $k$-Mahler power series can 
be decomposed as the product of a $k$-regular power series of a special type and 
the inverse of an infinite product of polynomials. Such a decomposition will play a key role in 
the proof of Theorem \ref{thm: main}.

\subsection{Automatic and regular power series} 

We recall here basic facts about regular power series, which were introduced by Allouche and Shallit \cite{AS2} (see also \cite{AS3} and \cite[Chapter 16]{AS}).  
They form a distinguished class of $k$-Mahler power series as well as a natural generalization of $k$-automatic power 
series. 

\medskip

A useful way to characterize $k$-automatic sequences, due to Eilenberg \cite{Eil}, is given in terms of the so-called $k$-kernel.  

\begin{defn} {\em Let $k\geq 2$ be an integer and let 
${\bf f} = (f(n))_{n\geq 0}$ be a sequence with values in a set $E$.  
The $k$-\emph{kernel} of ${\bf f}$ is defined as the set 
$$
\left\{ (f(k^{a}n+b))_{n\geq 0} \mid a \geq 0, b\in \{0,\ldots,a-1\} \right\} \, .
$$ 
  }
\end{defn}

\begin{thm}[Eilenberg] 
A sequence is $k$-automatic if and only if its $k$-kernel is finite.
\end{thm}

This characterization gives rise to the following natural generalization of automatic sequences introduced 
by Allouche and Shallit \cite{AS2}. 

\begin{defn} {\em Let $R$ be a commutative ring and let ${\bf f} = (f(n))_{n\geq 0}$ be a $R$-valued sequence. 
Then  ${\bf f}$ is said to be $k$-\emph{regular} 
if the dimension of the $R$-module spanned by its $k$-kernel is finite. }
\end{defn}

In the sequel, we will say that a power series $F(x)\in R[[x]]$ is $k$-\emph{regular} (respectively $k$-\emph{automatic}) 
if its sequence of coefficients is $k$-regular (respectively $k$-automatic).
In the following proposition, we collect some useful general facts about $k$-regular power series. 

\begin{prop}
\label{prop: reg2}
Let $R$ be a commutative ring and $k\geq 2$ be an integer. 
Then the following properties hold.
\begin{itemize}

\medskip

\item[{\rm (i)}] If $F(x)\in R[[x]]$ is $k$-regular and $I$ is an ideal of $R$, then $F(x)\bmod I \in (R/I)[[x]]$ is $k$-regular.

\medskip

\item[{\rm (ii)}] If $F(x)\in R[[x]]$ is $k$-regular, then the coefficients of $F(x)$ take only finitely many distinct values if and only if 
$F(x)$ is $k$-automatic.

\medskip

\item[{\rm (iii)}]  If $F(x)=\sum_{i\geq 0}f(i)x^{i}$ and $G(x)=\sum_{i\geq 0}g(i)x^{i}$ are  
two $k$-regular power series in $R[[x]]$, then the Cauchy product 
$$
F(x) G(x) := \sum_{i= 0}^{\infty} \left (\sum_{j=0}^i {i\choose j} f(j)g(i-j) \right)x^i
$$
is $k$-regular. 
\end{itemize}
 
\end{prop}

\begin{proof}
The property (i) follows directly from the definition of a $k$-regular sequence, while (ii) and (iii) correspond respectively  
to Theorem 16.1.5 and Corollary 16.4.2 in \cite{AS}. 
\end{proof}

In Section \ref{sec: elim}, we will need to use that $k$-regular sequences 
with complex values do have strict restrictions on the growth of their absolute values, 
a fact evidenced by the following result.

\begin{prop} Let $k\ge 2$ be a natural number and let $F(x)\in\mathbb{C}[[x]]$ be a $k$-regular power series.  
Then $F(x)$ is analytic in the open unit disc and there exist two positive real numbers $C$ and $m$ such that 
$$
\vert F(x) \vert < C (1 - \vert x \vert )^{-m} \, ,
$$
for all $x\in B(0,1)$.  
\label{thm: analyticunitdisc}
\end{prop}

\begin{proof} Let  $\displaystyle F(x)=\sum_{i=0}^{\infty} f(i)x^i\in \mathbb C[[x]]$ be a $k$-regular power series. 
Then there is some positive constant $A$ and some integer $d>0$ such that 
$$
\vert f(i) \vert \le A(i+1)^d \, ,
$$
for every nonnegative integer $i$ (see \cite[Theorem 16.3.1]{AS}).  
This immediately gives that $F(x)$ is analytic in the open unit disc.  
Moreover, for $x\in B(0,1)$,
$$
\vert F(x) \vert \le \sum_{i=0}^{\infty} A(i+1)^d |x|^i \le 
\sum_{i=0}^{\infty} Ad!{i+d\choose d}\vert x\vert^i = 
A d!(1-|x|)^{-d-1} \, .
$$
The result follows. 
\end{proof}


\subsection{Becker power series}

Becker \cite[Theorem 1]{Bec} showed that a $k$-regular power series  
is necessarily $k$-Mahler. In addition to this, he proved \cite[Theorem 2]{Bec} the following 
partial converse. The general converse does not hold.

\begin{thm}[Becker] 
Let $K$ be a field, let $k$ be a natural number $\ge 2$, 
and let $F(x)\in K[[x]]$
be a power series that satisfies a $k$-Mahler equation of the form
\begin{equation} \label{eq: becker} 
F(x)=\sum_{i=1}^n P_i(x)F(x^{k^i})
\end{equation} for some polynomials $P_1(x),\ldots ,P_n(x)\in K[x]$.  
Then $F(x)$ is a $k$-regular power series.
\label{thm: beck}
\end{thm}

\begin{defn}{\em
In honour of Becker's result, a power series $F(x)\in K[[x]]$ that satisfies an equation of the form given 
in Equation (\ref{eq: becker}) will be called a $k$-\emph{Becker} power series.}
\end{defn}

Theorem \ref{thm: beck} shows that the set of $k$-Becker power series is contained in the set of $k$-regular 
power series. However, the converse is not true.  As an example, we provide the following result that will also be 
used in Section \ref{sec: main}.

\begin{prop} Let $k$ be a natural number, and let $\omega\in \mathbb{C}$ be a root 
of unity with 
the property that if $j\ge 1$ then $\omega^{k^j}\neq \omega$.  Then
$$
\left(\prod_{j= 0}^{\infty} (1-\omega x^{k^j})\right)^{-1}
$$ 
is $k$-regular but is not $k$-Becker.
\label{prop: omega}
\end{prop}

\begin{proof}
Since $\omega$ is a root of unity, the sequence $\omega, \omega^k,\omega^{k^2},\ldots$ is eventually 
periodic and there is some smallest natural number $N$ such that 
$$
\omega^{k^{2N}}= \omega^{k^N} \,.
$$
Set $\beta := \omega^{k^N}$ and let us consider the polynomial
$$
Q(x)=(1-\beta x)(1-\beta x^k)\cdots (1-\beta x^{k^{N-1}})\, .
$$
Then
$$
\frac{Q(x^k)}{Q(x)} = \frac{1-\beta x^{k^N}}{1-\beta x} \, \cdot
$$  
Since
$$
1-\beta x^{k^N} = 1-(\beta x)^{k^N} \, ,
$$ 
we see that $Q(x^k)/Q(x)$ is a polynomial.  

Since 
$$
1-(\beta x)^{k^N} = \frac{Q(x^k)}{Q(x)} \cdot (1-\beta x) \, ,
$$
 we get that  
$(1-\omega x)$ divides the polynomial $Q(x^k)(1-\beta x)/Q(x)$.    
Furthermore, $(1-\omega x)$ cannot divide $(1-\beta x)$ since  by assumption 
$\omega \not= \beta$. By Euclid's lemma, we thus obtain that 
$$
\frac{Q(x^k)}{Q(x)} = (1-\omega x)S(x) 
$$ 
for some polynomial $S(x)$.

Set $G(x) :=Q(x)^{-1}F(x)$.  
Since $F(x)$ satisfies the $k$-Mahler recurrence
$$
F(x^k)=(1-\omega x)F(x) \, ,
$$ 
we see that
$$
G(x^k) =Q(x^k)^{-1}(1-\omega x)Q(x) G(x),
$$ 
or equivalently,
$$
G(x) = S(x)G(x^k) \, .
$$
Thus $G(x)$ is a $k$-Becker power series. By Proposition \ref{prop: reg2}, $F(x)$ is $k$-regular as it is a product of a 
polynomial (which is $k$-regular) and a $k$-regular power series.

\medskip

On the other hand, $F(x)$ cannot be a $k$-Becker power series.  To see this, suppose that $F(x)$ satisfies an equation 
of the form
$$
F(x)=\sum_{i=1}^d P_i(x)F(x^{k^i}) \,.
$$  
Now, dividing both sides by $F(x^k)$, the right-hand side becomes a polynomial in $x$, while the left-hand side is 
$(1-\omega x)^{-1}$, a contradiction.  The result follows. 
\end{proof}

In Section \ref{sec: elim}, we will need the following basic result about $k$-Becker power series.

\begin{lem}
\label{lem: constant} 
Let $k\ge 2$ and let us assume that $F(x)\in K[[x]]$  satisfies a $k$-Mahler equation of the form
$$
F(x)=\sum_{i=1}^n a_i F(x^{k^i})
$$ 
for some constants $a_1,\ldots ,a_n\in K$.  Then $F(x)$ is constant.
\end{lem}

\begin{proof}
Let us denote by 
$F(x)=\sum_{i\geq 0} f(i) x^{i}$
the power series expansion of $F(x)$.  If $F(x)$ were non-constant, there would be some smallest positive integer $i$ 
such that $f(i)\neq 0$. Thus $F(x) = \lambda + x^{i}F_0(x)$ for some $\lambda$ in $K$ and some $F_0(x)\in K[[x]]$. 
But taking the coefficient of $x^i$ in the right-hand side of the equation
$$
F(x)=\sum_{i=1}^n a_i F(x^{k^i}) \, ,
$$
we see that $f(i)=0$, a contradiction.  
The result follows.
\end{proof}

Though there are some Mahler functions that are not Becker functions, the following result shows that every $k$-Mahler power series can 
be decomposed as the product of a $k$-Becker power series and the inverse of an infinite product of polynomials. 
This decomposition will turn out to be very useful to prove Theorem \ref{thm: main}.  
We note that a similar result also appears as Theorem 31 in the 
Ph.\@ D.\@ Thesis of Dumas \cite{Dum}. 

\begin{prop} \label{rem: decomp}
Let $k$ be a natural number, let $K$ be a field, and let 
$F(x)\in K [[x]]$ be a $k$-Mahler power series satisfying 
an equation of the form 
$$
\sum_{i=0}^n P_i(x)F(x^{k^{i}}) = 0 \, ,
$$
where $P_0(x),\ldots, P_n(x)\in K[x]$ and  $P_0(0)=1$. 
Then there is a $k$-Becker power series $G(x)$ such that
$$
F(x) \ =  \ \left(\prod_{i= 0}^{\infty} P_0(x^{k^i})\right)^{-1}G(x) \, .
$$
\end{prop}

\begin{proof}
Since $P_0(0)=1$, the infinite product
$$
H(x) :=\prod_{i=0}^{\infty} P_0(x^{k^i}) 
$$ 
converges to an invertible element of $K[[x]]$. 
By definition, $H(x)$ satisfies the following  equation: 
$$ 
H(x)=P_0(x)H(x^k)
$$ 
and hence $H(x)$ is a $k$-Becker power series.
Now, set $G(x):=H(x)F(x)$. Then our assumption on $F(x)$ implies that  
$$
\sum_{i=0}^n P_i(x)H(x^{k^i})^{-1}G(x^{k^i}) \ = \ 0 \, .
$$  
Dividing both sides by $P_0(x)H(x)^{-1}$, we obtain that 
$$
G(x) = - \sum_{i=1}^n P_i(x)\left(\prod_{j=1}^{i-1}P_0(x^{k^i})\right)  G(x^{k^i})\,.
$$ 
This shows that $G$ is a $k$-Becker power series. Hence $F(x)$ can be written as 
$$
F(x) \ =  \ \left(\prod_{i= 0}^{\infty} P_0(x^{k^i})\right)^{-1}G(x)\, ,
$$
where $G(x)$ is a $k$-Becker power series. 
This ends the proof. 
\end{proof}



\section{Conditions on $k$ and $\ell$}\label{sec: kl}

In this section, $K$ will denote an arbitrary field. We consider power series in $K[[x]]$ 
that are both $k$- and $\ell$-Mahler with respect to two multiplicatively independent natural numbers 
$k$ and $\ell$.  More specifically, we look at the set of natural numbers $m$ 
for which such a power series is necessarily $m$-Mahler.

\begin{prop} 
Let $k$ and $\ell$ be two multiplicatively independent natural numbers and 
let $F(x)\in K[[x]]$ be a power series that is both $k$- and $\ell$-Mahler.   
Let us assume that $a$ and $b$ are integers with the property that 
$m:=k^a\ell^b$ is an integer greater than $1$.  
Then $F(x)$ is also $m$-Mahler.
\label{prop: k'l'}
\end{prop}

\begin{proof} 
Let $V$ denote the $K(x)$-vector space spanned by all the power series 
that belong to the set $\left\{F(x^{k^a\ell^b}) \mid a,b \in\mathbb N\right\}$.  
By assumption, there exists some natural number $N$ such that 
$F(x^{k^n})\in \sum_{i=0}^{N-1} K(x)F(x^{k^i})$ and
$F(x^{\ell^n})\in \sum_{i=0}^{N-1} K(x)F(x^{\ell^i})$ for every integer $n\ge N$.  
Thus $V$ is a $K(x)$-vector space of dimension at most $N^2$.  

\medskip

Suppose that $a$ and $b$ are integers such that 
$m:=k^a\ell^b$ is an integer greater than $1$.  If $a$ and $b$ are nonnegative, then $F(x^{m^j})\in V$ for every integer 
$j\ge 0$ and since the dimension of $V$ is a finite, we see that $F(x)$ is $m$-Mahler.  
Thus we may assume that at least one of $a$ or $b$ is negative.  
Since $m\geq 1$, at least one of $a$ or $b$ must also be positive.  
Without loss of generality, we may thus assume that $a>0$ and $b<0$.  

\medskip

We are now going to show that $F(x^{m^j})\in V$ for every nonnegative integer $j$.   
To see this, we fix a nonnegative integer $j$.  Then we observe that $m^j\ell^{-bj}=k^{ja}$ and thus 
$F(x^{m^jl^{i}})$ belongs to $V$ for every integer $i\geq -bj$. 
Since $-bj\geq 0$, there exists a smallest nonnegative integer $i_0$ 
such that  $F(x^{m^j \ell^i})\in V$ for every integer $n\ge i_0$.  
If $i_0$ is zero, then we are done.  
We assume that $i_0$ is positive and look for a contradiction. By definition of $i_0$, we note that   
 $F(x^{m^j \ell^{i_0-1}})\not\in V$.  By assumption, $F(x)$ satisfies a $\ell$-Mahler equation of the form
$$
\sum_{i=0}^N P_i(x)F(x^{\ell^i}) \ = \ 0 \,,
$$ 
with $P_0(x),\ldots ,P_N(x)\in K[x]$ and $P_0(x)
\neq 0$.  Applying the operator $x \mapsto x^{m^j\ell^{i_0-1}}$, we get that
$$
P_0(x^{m^j\ell^{i_0-1}})F(x^{m^j\ell^{i_0-1}}) = -\sum_{i=1}^N P_i(x^{m^j\ell^{i_0-1}})F(x^{m^j\ell^{i_0-1+i}}) \, .
$$  
By definition of $i_0$, the right-hand side of this equation is in $V$, and so $F(x^{m^j\ell^{i_0-1}})\in V$ since $P_0(x)$ is nonzero. 
This is a contradiction.  It follows that 
$F(x^{m^j})\in V$ for every nonnegative integer $j$. 

\medskip

Since $V$ is a  $K(x)$-vector space of dimension at most $N^2$, we see that 
$F(x),F(x^{m}),\ldots ,F(x^{m^{N^2}})$ are linearly dependent over $K(x)$, which implies that $F(x)$ is $m$-Mahler. 
This ends the proof. 
\end{proof}

\begin{cor} 
Let $k$ and $\ell$ be two multiplicatively independent natural numbers and 
let $F(x)\in K[[x]]$ be a power series that is both $k$- and $\ell$-Mahler.     
Then there exist two multiplicatively independent positive integers $k'$ and $\ell'$ 
such that the following conditions hold.

\medskip

\begin{enumerate}
\item[{\rm (i)}] There is a prime number $p$ that divides $k'$ and does not divide $\ell'$.

\medskip

\item[{\rm (ii)}] There is a prime number $q$ that divides $\ell'$ and does not divide $k'$.

\medskip

\item[{\rm (iii)}] $F(x)$ is both $k'$- and $\ell'$-Mahler.

\end{enumerate}
\label{cor: pqkl}
\end{cor}

\begin{proof} There exist prime numbers $p_1,\ldots ,p_m$ and nonnegative integers 
$a_1,\ldots ,a_m,b_1,\ldots ,b_m$ such that 
$$
k=\prod_{i=1}^m p_i^{a_i} \;\mbox{ and }\; \ell=\prod_{i=1}^m p_i^{b_i} \, .
$$
Moreover, we can assume that, for each $i$, at least one of $a_i$ or $b_i$ is positive.

Note that if there are $i$ and $j$ such that $a_i=0$ and $b_j=0$, then we can take $k':=k$ and $\ell':=\ell$ and set 
$p:=p_j$ and $q:=p_i$ to obtain the desired result.  
Thus we can assume without loss of generality that $b_i>0$ for $i\in \{1,\ldots ,m\}$.  
Then there is some $i_0\in \{1,\ldots ,m\}$ such that $a_{i_0}/b_{i_0} \le a_j/b_j$ for all $j\in \{1,\ldots ,m\}$.
In particular, $c_j := a_jb_{i_0}-b_ja_{i_0}$ is a nonnegative integer for all $j\in \{1,\ldots ,m\}$. 
Hence 
$$
k':=k^{b_{i_0} }\ell^{-a_{i_0} } = \prod_{j=1}^m p_j^{c_j}\in \mathbb{N} \, .
$$
Furthermore, $p_{i_0}$ does not divide $k'$ and since $k$ and $\ell$ are 
multiplicatively independent, the $c_i$'s are not all equal to zero. 

Now we pick $i_1\in \{1,\ldots ,m\}$ such that $c_{i_1}/b_{i_1}\ge c_j/b_j$ for all $j\in \{1,\ldots ,m\}$.  
Note that $c_{i_1}>0$ since the $c_i$'s are not all equal to zero.  
Set 
$$
\ell':=\ell^{c_{i_1}} (k')^{-b_{i_1}} = \prod_{j=1}^m p_j^{b_jc_{i_1}-b_{i_1}c_j} \in \mathbb N \, .
$$  
Since $c_{i_0}=0$, $c_{i_1}>0$ and the $b_i$'s are positive, we get that $p_{i_0}$ divides $\ell'$.  
Moreover, $p_{i_1}$ does not divide $\ell'$ while $p_{i_1}$ divides $k'$ for $c_{i_1}$ is positive.   
In particular, $k'$ and $\ell'$ are multiplicatively independent.   
 Furthermore, Proposition \ref{prop: k'l'} implies that $F(x)$ is both $k'$- and $\ell'$-Mahler. 
 Setting $q:=p_{i_0}$ and $p=p_{i_1}$,  we obtain that $k'$ and $\ell'$ have all the desired properties. 
 This concludes the proof.
\end{proof}



\section{Elimination of singularities at roots of unity}\label{sec: elim}

In this section we look at the singularities of $k$-Mahler functions at roots of unity.  
Strictly speaking, we do not necessarily eliminate singularities, 
and so the section title is perhaps misleading.  We do, however, show that one can 
reduce to the case of considering Mahler equations whose singularities at roots of 
unity have a restricted form.

\begin{notn} {\em Throughout this section we make the following assumptions and use the following notation.

\medskip

\begin{itemize}
\item[(i)] We assume that $k$ and $\ell$ are two multiplicatively independent natural numbers.

\medskip

\item[(ii)] We assume there exist primes $p$ and $q$ such that $p|k$ and $p$ does not divide $\ell$ 
and such that $q|\ell$ and $q$ does not divide $k$.

\medskip

\item[(iii)] We assume that $F(x)$ is a $k$-Mahler complex power series that satisfies an equation of the form 
$$
\sum_{i=0}^d P_i(x)F(x^{k^i}) \ = \ 0
$$ 
with $P_0,\ldots ,P_d\in \mathbb{C}[x]$ and $P_0(0)\neq 0$.

\medskip

\item[(iv)] We assume that $F(x)$ is an $\ell$-Mahler complex power series that satisfies an equation of the form 
$$
\sum_{i=0}^e Q_i(x)F(x^{\ell^i}) \ = \ 0
$$ 
with $Q_0,\ldots ,Q_e\in \mathbb{C}[x]$ and $Q_0(0)\neq 0$.
\end{itemize}\label{notn: 1}}
\end{notn}

In this section, our aim is to prove the following result.

\begin{thm}
Let $F(x)\in\mathbb C[[x]]$ be a power series that satisfies Assumption-Notation \ref{notn: 1} 
and  that  is not a polynomial. 
Then $F(x)$ satisfies a non-trivial $k$-Mahler equation of the form 
$$
\sum_{i=0}^d P_i(x)F(x^{k^i}) \ = \ 0
$$ 
with the property that $P_0(0)=1$ and $P_0(\alpha)\neq 0$ if $\alpha$ is a root of unity satisfying $\alpha^{k^j}=\alpha$ 
for some positive integer $j$.
\label{thm: elim}
\end{thm}


\subsection{Asymptotic estimates for some infinite products}

We first study the behaviour around the unit circle of infinite products of the form 
$$
\left(\prod_{i= 0}^{\infty} P(x^{k^i})\right)^{-1} \, ,
$$
where $P(x)\in \mathbb C[x]$ and $P(0)=1$. 

\begin{lem}\label{lem: last}
Let $k\ge 2$ be a natural number. Then
$$
\lim_{\underset{0<t<1}{ t \to 1}}  \left(\prod_{j= 0}^{\infty} \frac{1}{1- t^{k^j}}\right) \cdot (1-t)^A = \infty \, ,
$$ 
 for every positive real number $A$.
\end{lem}

\begin{proof}
Let $t$ be in $(1-1/k^9,1)$. Let $N\geq 2$ be the largest natural number such that $t\in(1-k^{-(N+1)^2},1)$.  
Then
\begin{eqnarray*}
  \prod_{j= 0}^{\infty} (1-t^{k^j})^{-1}
  &\ge & \prod_{j= 0}^{N} (1-t^{k^j})^{-1} \\
  &=& (1-t)^{-(N+1)} \prod_{j=0}^N (1+t+\cdots +t^{k^j-1})^{-1} \\
  &\ge & (1-t)^{-(N+1)} \prod_{j=0}^N k^{-j} \\
  &\ge & (1-t)^{-(N+1)} k^{-(N+1)^2} \\ 
  &> &  (1-t)^{-N} \, .  
\end{eqnarray*} 
By definition of $N$, we obtain that $t<1-k^{-(N+2)^2}$, which easily gives that 
$$
N> \sqrt{\frac{-\log (1-t)}{4\log k}} \,\cdot
$$
This ends the proof for the right-hand side tends to infinity when $t$ tends to $1$.  
\end{proof}

\begin{lem} \label{lem: xxx}
Let $k\ge 2$ be a natural number.  Then for $t\in (0,1)$, we have
$$
\sum_{i=1}^{\infty} t^i/i  \ \ge \ (1-1/k)\sum_{i=0}^{\infty} t^{k^i} \, .
$$
\end{lem}

\begin{proof} We have
\begin{eqnarray*}
 \sum_{i=1}^{\infty} t^{i}/i &=& t + \sum_{i=0}^{\infty} \sum_{j=k^i+1}^{k^{i+1}} t^j/j \\
 &\ge & t +  \sum_{i=0}^{\infty} \sum_{j=k^i+1}^{k^{i+1}}  t^{k^{i+1}}/k^{i+1} \\
 &=& t+\sum_{i=0}^{\infty} t^{k^{i+1}}(k^{i+1}-k^i)/k^{i+1} \\
 &=& t + (1-1/k)\sum_{i=0}^{\infty} t^{k^{i+1}}\\
 &\ge & (1-1/k)\sum_{i=0}^{\infty} t^{k^{i}} \, ,
 \end{eqnarray*}
which ends the proof.
\end{proof}

 \begin{lem}
 \label{lem: wai}
  Let $k\ge 2$ be a natural number and let $\lambda \neq 1$ be a complex number.  
  Then there exist two positive real numbers $A$ and $\varepsilon$ such that
  $$
  (1-t)^A \ < \  \left|\prod_{j= 0}^{\infty} \frac{1}{1-\lambda t^{k^j}}\right| \ < \ (1-t)^{-A}
  $$  
  whenever $1-\varepsilon<t<1$. 
  \end{lem}

 \begin{proof}  We first prove the inequality on the right-hand side. 
 
 Note that since $\lambda\not=1$ there exist two real numbers $\varepsilon_0$ and $c_0$, $c_0<1$, such that
 \begin{equation}
 \label{eq: inf}
 \inf \left\{ \left\vert 1 - \lambda t^{k^j} \right\vert \mid t\in (1-\varepsilon_0,1), j \geq 0\right\} > c_0 \, .
\end{equation} 
 Let $t\in (1-\varepsilon_0,1)$ and let $N$ be the largest nonnegative integer such that 
 $t^{k^N} \ge 1/2$.  Then for $j\ge 1$ we have 
$t^{k^{N+j}} = (t^{k^{N+1}})^{k^{j-1}} <   (1/2)^{k^{j-1}}$. 
Hence
$$
\left|1-\lambda t^{k^{N+j}}\right | \ge 1-  \vert \lambda\vert (1/2)^{k^{j-1}} \, .
$$
Since the series $\sum_{j\ge 0} (1/2)^{k^{j-1}}$ converges, we get that  the infinite product 
$$
\prod_{j= 0}^{\infty} \left| \frac{1}{1-\lambda t^{k^{N+j} }} \right|
$$ 
converges to some positive constant $c_1$.  Then
\begin{eqnarray*}  \left|\prod_{j=0}^{\infty} (1-\lambda t^{k^j})^{-1} \right|
  &=& \prod_{j= 0}^N \left|1-\lambda t^{k^j}\right|^{-1} \prod_{j= 1}^{\infty} \left|1-\lambda t^{k^{N+j}} \right|^{-1} \\ \\
  &\le & (1/c_0)^{N+1} c_1\\ \\
  &=& (k^{N+1})^{-\log c_0/\log k} c_1\, .
  \end{eqnarray*}
Furthermore, we have by assumption that $t^{k^{N+1}}<1/2$ and thus $k^{N+1}< -\log 2 /\log t$. 
This implies that 
 $$
  \left|\prod_{j=0}^{\infty} (1-\lambda t^{k^j})^{-1} \right| \leq c_1 \left(-\log 2 /\log t\right)^{-\log c_0/\log k} \, .
 $$
On the other hand, note that $\lim_{t\to 1} (1-t)/\log(t)=-1$ and hence there exists 
some positive $\varepsilon < \varepsilon_0$  such that
$$ 
c_1 \left(-\log 2 /\log t\right)^{-\log c_0/\log k} < c_1 \left(2\log 2 (1-t)\right)^{\log c_0/\log k}  \, ,
$$ 
whenever $t\in (1-\varepsilon,1)$.  
Since $c_0<1$, we obtain that there exists a positive real number $A_1$ such that 
$$
  \left|\prod_{j= 0}^{\infty} (1-\lambda t^{k^j})^{-1} \right| < (1-t)^{-A_1} \, ,
 $$
for all $t\in (1-\varepsilon,1)$. This gives the right-hand side bound in the statement of the lemma.  
  
  \medskip
  
  To get the left-hand side, note that for all $t\in (0,1)$, 
  $$
  \left|\prod_{j= 0}^{\infty} \frac{1}{1-\lambda t^{k^j}}\right|\ge \prod_{j= 0}^{\infty} (1+\vert \lambda\vert t^{k^j})^{-1}
   \ge \prod_{j= 0}^{\infty} \exp(-|\lambda|t^{k^j}) \, .
  $$  
 By Lemma \ref{lem: xxx},  we have
 $$
 \prod_{j=0}^{\infty} \exp(-|\lambda|t^{k^j})\ge \exp\left(-|\lambda|(1-1/k)^{-1}\sum_{i= 1}^{\infty} t^i/i\right)
 = (1-t)^{\vert\lambda\vert k/(k-1)} \,.
 $$ 
We thus obtain that, for all $t\in (0,1)$,  
$$
\left|\prod_{j= 0}^{\infty} \frac{1}{1-\lambda t^{k^j}} \right| > (1-t)^{A_2} \,,
$$ 
where $A_2 := \lfloor \vert\lambda\vert k/(k-1) \rfloor + 1$.  
Taking $A$ to be equal to the maximum of $A_1$ and $A_2$, we get the desired result. 
  \end{proof}

\begin{cor}
 \label{cor: mog}
  Let $k\ge 2$ be a natural number, let $\alpha$ be root of unity that satisfies $\alpha^{k}=\alpha$, and let 
  $P(x)$ be a nonzero polynomial with $P(0)=1$ and $P(\alpha)\neq 0$.  
  Then there exist two positive real numbers $A$ and $\varepsilon>$ such that
  $$
  (1-t)^A\ < \  \left| \left(\prod_{j= 0}^{\infty} P((t\alpha)^{k^j}))\right)^{-1}\right| \ < \ (1-t)^{-A}
  $$  
  whenever $1-\varepsilon<t<1$. 
\end{cor}
  
  \begin{proof}
  Let $\beta_1,\ldots, \beta_s$ denote the complex roots of $P$ (considered with muliplicities) so that 
 we may factor $P(x)$ as $P(x)=(1-\beta_1^{-1}x)\cdots (1-\beta_s^{-1}x)$. We thus obtain 
 $$
\left|\prod_{j=0}^{\infty} \frac{1}{P((t\alpha)^{k^j})}\right| =  \prod_{i=1}^s \left|\prod_{j= 0}^{\infty} \frac{1}{1-\beta_i^{-1}\alpha t^{k^j}}\right| \, ,
 $$
 where $\beta_i^{-1}\alpha \not= 1$ for every $i\in\{1,\ldots,s\}$.  
 Then by Lemma \ref{lem: wai}, there are natural numbers $A_i$  and positive real numbers $\varepsilon_i$, $0<\varepsilon_i<1$, such that
  $$
  (1-t)^{A_i} \ < \  \left|\prod_{j= 0}^{\infty} (1-\beta_i^{-1}\alpha t)^{-1}\right| \ < \ (1-t)^{-A_i}
  $$
  whenever $1-\varepsilon_i<t<1$.  Taking $\varepsilon := \min(\varepsilon_1,\ldots, \varepsilon_s)$ and $A:=\sum_{i=1}^s A_i$, 
  we obtain the desired result.
  \end{proof}
  

\subsection{Asymptotic estimates for Becker functions}  
  
  We are now going to provide asymptotic estimates for Becker functions. 
  We denote by $\Vert \cdot \Vert$ a norm on $\mathbb C^d$. 
  We let $B(x,r)$ (respectively $\overline{B(x,r)}$) denote 
  the open (respectively closed) ball of center $x$ and radius $r$. 
  Our results will not depend on the choice of this norm. 
  
\begin{lem}\label{lem: kappa}
Let $d$ and $k$ be two natural numbers, $\alpha$  a root of unity such that $\alpha^{k}=\alpha$, and  
$A:\overline{B(0,1)}\to M_d(\mathbb{C})$  a continuous matrix-valued function.   
Let us assume that $w(x)\in \mathbb{C}[[x]]^d$ satisfies the equation 
$$
w(x)=A(x)w(x^k)
$$ 
for all $x\in  B(0,1)$. Let us also assume that the following properties hold.
\medskip

\begin{itemize}
 
 \item[(i)] The coordinates of $w(x)$ are analytic in $B(0,1)$ and continuous on  $\overline{B(0,1)}$.
 
 \medskip
 
 \item[(ii)] The matrix $A(\alpha)$ is not nilpotent. 
 
\medskip

\item[(iii)] The set $\left\{w(x)~\mid~x\in B(0,1)\right\}$ is not contained in a proper subspace 
of $\mathbb{C}^d$. 

\end{itemize}

\medskip

Then there exist a positive real number $C$ and a subset $S\subseteq (0,1)$ that has $1$ as a limit point 
such that
$$
||w(t\alpha)|| > (1-t)^{C}
$$ 
for all $t\in S$.
\end{lem}

\begin{proof}  
Since $A(\alpha)$ is not nilpotent, there is some natural number $e$ such that the kernel of $A(\alpha)^e$ and the kernel of 
$A(\alpha)^{e+1}$ are equal to a same proper subspace of $\mathbb{C}^d$, say $W$.  
Then there is a nonzero subspace $V$ such that $A(\alpha)(V)\subseteq V$ and 
$V\oplus W=\mathbb{C}^d$. 
Moreover, by compactness, there is a positive real number $c_0$, $c_0<1$, such that 
\begin{equation}\label{eq: minA}
||A(\alpha)(w)||\ge c_0
\end{equation} 
whenever $w\in V$ is a vector of norm $1$.  

Since every vector has a unique decomposition as a sum of elements from $V$ and $W$, 
we have a continuous linear projection 
map $\pi:\mathbb{C}^d\to V$ with the property that $u-\pi(u)\in W$ for all $u\in \mathbb{C}^d$.  
We infer from Inequality (\ref{eq: minA}) that 
\begin{equation}\label{eq: minpi}
||\pi(A(\alpha)(u))||= \Vert A(\alpha)(\pi(u))\Vert \ge c_0||\pi(u)||
\end{equation}
for all $u\in \mathbb{C}^d$. 
Since $A$ is continuous on $\overline{B(0,1)}$, Inequality (\ref{eq: minpi}) 
implies the existence of a positive constant 
$\varepsilon>0$ such that
$$
|\pi(A(x)(u))|| > c_0||\pi(u)||/2 \, ,
$$
for all $u\in \mathbb{C}^d$ and all $x\in B(\alpha,\varepsilon)\cap \overline{B(0,1)}$. 
It follows by a simple induction that if $x_1,\ldots ,x_m\in B(\alpha,\varepsilon)\cap \overline{B(0,1)}$ then
\begin{equation}\label{eq: minind}
||\pi(A(x_1)\cdots A(x_m)(u))|\ge (c_0/2)^m ||\pi(u)|| \, .
\end{equation}

Set  
$$
{\mathcal Z} :=\left\{t\in [0,1-\varepsilon/2]~\mid~w(t\alpha)\in W)\right\} \, .
$$ 
We claim that ${\mathcal Z}$ is a finite set. Otherwise, there would be a nonzero row vector $u$ such that 
$u\cdot w(t\alpha)=0$ for infinitely many $t\in [0,1-\varepsilon/2]$.  But $u\cdot w(x)$ is analytic in $B(0,1)$ 
for $w(x)$ is and hence 
it would be identically zero on $B(0,1)$ by the identity theorem. 
This would contradict the assumption that $\{w(x)~\mid~x\in  B(0,1)\}$ 
is not contained in a proper subspace of $\mathbb{C}^d$.   

\medskip

Let us pick a sequence $t_0,t_1,t_2,\ldots $ in $(0,1)$ such that:
\begin{enumerate}
\item[$\bullet$] $t_i^k=t_{i-1}$ for $i\ge 1$;
\item[$\bullet$] $t_0\in (1-\varepsilon, 1-\varepsilon/2)$;
\item[$\bullet$] $t_0\not \in \mathcal Z$.
\end{enumerate}
Note that $t_n\to 1$ as $n\to \infty$.
Since $\mathcal Z$ is finite, there is an open neighbourhood $U\subseteq [0,1]$ of $\mathcal Z$ such that 
$t_0\not\in U$. Set
$X:=[0,1-\varepsilon/2]\setminus U$.  Then $X$ is compact and $||\pi(w(x\alpha)||$ is nonzero for $x\in X$.  
Thus there exists a positive real number $c_1$ such that 
$$
||\pi(w(x\alpha))|| >  c_1
$$ for all $x\in X$.  
Then we infer from (\ref{eq: minind}) that 
\begin{eqnarray*}
||\pi(w(t_n\alpha ))|| &= & ||\pi(A(t_n\alpha)A(t_{n-1}\alpha)\cdots A(t_1\alpha)(w(t_0\alpha))||\\ \\
&\ge & (c_0/2)^n||\pi(w(t_0\alpha)|| \\ \\
& > &  c_1(c_0/2)^n \, .
\end{eqnarray*}
Furthermore,  since the projection $\pi$ is continuous, there is some positive real number $c_2$ such that
$||\pi(u)||<c_2||u||$ for all $u\in \mathbb{C}^d$.
Thus 
$$
||w(t_n)||\ge c_2^{-1}||\pi(w(t_n))|| > c_2^{-1}c_1(c_0/2)^n
$$ 
for all $n\geq 1$.  

On the other hand, since 
$$
\lim_{a\to 0^+} \frac{t_0^a - 1}{a}= \log(t_0) < 0\, ,
$$ 
there exists some $\varepsilon_0\in (0,1)$ such that 
$$
t_0^{a} < 1+a\log(t_0)/2 < 1-a(1-t_0)/2 
$$ 
for $a\in (0,\varepsilon_0)$.  Thus if $n$ is large enough, say $n\geq n_0$, then $k^n>1/\varepsilon_0$ and  we have
$t_n = (t_0)^{1/k^n} <1-(1-t_0)/(2k^n)$. Hence $k^n >(1-t_0)/(2(1-t_n))$. 
 Then we have
\begin{eqnarray*}
||w(t_n\alpha)|| &> & c_2^{-1}c_1(c_0/2)^n \\ \\
&=& c_2^{-1}c_1 k^{n\log_k (c_0/2)}  \\ \\
&\ge & \left( c_2^{-1}c_1 \left(\frac{(1-t_0)}{2}\right)^{\log_k(c_0/2)} \right) (1-t_n)^{-\log_k(c_0/2)}.
\end{eqnarray*} 
Thus if we take $C:= -2\log_{k}(c_0/2)>0$, the fact that $t_n$ tends to $1$ as $n$ tends to infinity implies the existence of 
a positive integer $n_1\geq n_0$ such that 
$$
||w(t_n\alpha)|| > (1-t_n)^{C} \, ,
$$
for all $n\geq n_1$. Taking $S:= \left\{t_n \in (0,1) \mid n\geq n_1\right\}$, 
we obtain the desired result.
\end{proof}

\begin{lem}\label{lem: B} 
Let  
$B:  \overline{B(0,1)}\to M_d(\mathbb{C})$ be a continuous matrix-valued function 
whose entries are analytic inside the unit disc and continuous on the closed unit disc.     
Let us assume that there exist two positive real numbers $\varepsilon$ and $M$ such that 
 $|\det(B(x))|> (1-|x|)^M$ for every $x$ such that $1-\varepsilon<\vert x\vert <1$. 
Then there exists a positive real number $C$  such that 
for every column vector $u$ of norm $1$, 
we have 
$$
||B(x)(u)||\ge (1-|x|)^{C}
$$ 
for every  $x$ such that $1-\varepsilon<\vert x\vert <1$.
\end{lem}

\begin{proof} 
Our assumption implies that $B(x)$ is invertible for every $x$ such that $1-\varepsilon<\vert x\vert <1$. 
Let $\Delta(x)$ denote the determinant of $B(x)$.  Using the classical adjoint formula for the inverse of $B(x)$, 
we see that $B(x)^{-1}$ has entries $c_{i,j}(x)$ that have the property that they are expressible (up to sign) as the ratio of 
the determinant of a submatrix of $B(x)$ and $\Delta(x)$.  
Since the entries of $B(x)$ are continuous on $\overline{B(0,1)}$, 
each determinant of a submatrix of $B(x)$ is also continuous.  
By compactness, we see that there is a positive real number $\kappa$ 
such that  
$$
|c_{i,j}(x)| \le \kappa/|\Delta(x)| \le \kappa (1 -|x|)^{-M}
$$ 
for every $(i,j)\in\{1,\ldots,d\}^2$ and every $x$ such that $1-\varepsilon<\vert x\vert <1$.    
Thus there exists a positive real number $C$ such that 
$$
\Vert B(x)^{-1}\Vert \le (1-\vert x \vert)^{-C}
$$ 
for every $x$ such that $1-\varepsilon<\vert x\vert <1$. 
It follows that if $u$ is a vector of norm $1$, then 
$$
\Vert B(x)(u) \Vert\ge (1-|x|)^{C}.
$$ 
for every $x$ such that $1-\varepsilon<\vert x\vert <1$.  The result follows.
\end{proof}

\begin{cor}
\label{cor: kappa}
Let $d$ and $k$ be two natural numbers, $\alpha$  be a root of unity such that $\alpha^{k}=\alpha$, $\zeta$ be 
a root on unity such that $\zeta^{k^j}=1$ for some natural number $j$, and  
$A:\overline{B(0,1)}\to M_d(\mathbb{C})$  be a continuous matrix-valued function.   
Let us assume that $w(x)\in \mathbb{C}[[x]]^d$ satisfies the equation 
$$
w(x)=A(x)w(x^k)
$$ 
for all $x\in B(0,1)$. Let us also assume that the following properties hold. 

\medskip

\begin{itemize}
 
 \item[(i)] The coordinates of $w(x)$ are analytic in $B(0,1)$ and continuous on  $\overline{B(0,1)}$.
 
 \medskip
 
 \item[(ii)] The matrix $A(\alpha)$ is not nilpotent. 
 
\medskip

\item[(ii)] There exists two positive real numbers $\varepsilon$ and $M$ such that  
$|\det(A(x))|>(1-|x|)^M$ for every 
$x$ with $1-\varepsilon<\vert x\vert <1$.

\medskip

\item[(iv)] The set $\left\{w(x)~\mid~x\in  B(0,1)\right\}$ is not contained in a proper subspace 
of $\mathbb{C}^d$. 

\end{itemize}

\medskip

Then there exist a positive real number $C$ and a subset $S\subseteq (0,1)$ that has $1$ as a limit point 
such that
$$
||w(t\alpha\beta)|| > (1-t)^{C}
$$ 
for all $t\in S$.
\end{cor}

\begin{proof} 
Since $A(\alpha)$ is not nilpotent, we first infer from Lemma \ref{lem: kappa} that there exist a positive real number 
$C_0$ and a sequence $t_n\in (0,1)$, which tends to $1$, such that
$||w(t_n\alpha)||> (1-t_n)^{C_0}$ for every integer $n\ge 1$.  Let $s_n\in (0,1)$ be such that $s_n^{k^j}=t_n$.
Then 
$$
w(s_n\alpha\beta)=A(s_n\alpha\beta)A(s_{n}^k\alpha\beta^k)\cdots A(s_{n}^{k^{j-1}}\alpha\beta^{k^{j-1}})(w(t_n\alpha))\, .
$$
By assumption there exists a positive real number $M$ such that 
 $|\det(A(x))|>(1-|x|)^M$ for every 
$x$ with $1-\varepsilon<\vert x\vert <1$. 
Set 
$$
B(x):=A(x\alpha\beta)A(x^k\alpha\beta^k)\cdots A(x^{k^{j-1}}\alpha\beta^{k^{j-1}}) \, .
$$  
Then  there is a positive real number $C_1$ such that if $(1-\varepsilon)^{1/k^{j-1}}<|x|<1$ then
$$
\det(B(x))> (1-|x|)^M\cdots (1-|x|^{k^{j-1}})^M \geq (1-|x|)^{jM} \, .
$$ 
It follows from Lemma \ref{lem: B} that there exists a positive real number $C_1$ such that for $n$ sufficiently large we have
\begin{eqnarray*}
||w(s_n\alpha\beta)|| & = & ||B(s_n)(w(t_n\alpha))|| > (1-s_n)^{C_1}||w(t_n\alpha)|| \\
&>& (1-s_n)^{C_1}(1-t_n)^{C_0} \, .
\end{eqnarray*}
Since $(1-t_n)/(1-s_n)\to k^j$ as $n\to \infty$, we see that if we take 
$C:=2(C_1+C_0)$ then we have
$$
||w(s_n\alpha\beta)||\ge  (1-s_n)^{C}
$$ 
for all $n$ sufficiently large.  The result follows.
\end{proof}


We are now almost ready to prove the main result of this section.   
Before doing this, we give the  following simple lemma.

\begin{lem} \label{lem: nilp} Let $d$ be a natural number and let $A$ be a $d\times d$ complex matrix whose $(i,j)$-entry is 
$\delta_{i,j+1}$ if $i\ge 2$.  If there is an integer $r$ such that the $(1,r)$-entry of of $A$ is nonzero, then $A$ is not nilpotent.
\end{lem}

\begin{proof} Let $(a_1,\ldots ,a_d)$ denote the first row of $A$.  Then by the theory of companion matrices, $A$ 
has characteristic polynomial $x^d-a_1 x^{d-1}-a_2 x^{d-2}-\cdots - a_d$.  But if $A$ is nilpotent, its characteristic 
polynomial must be $x^d$ and hence the first row of $A$ must be zero.
\end{proof}


\subsection{Proof of Theorem \ref{thm: elim}}

We are now ready to prove the main result of this section. 

\begin{proof}[Proof of Theorem \ref{thm: elim}] Consider the set $I$ of all polynomials $P(x)\in\mathbb C[x]$ for which 
there exist positive integers $a$ and $b$ with $a<b$ such that
$$
P(x)F(x)\in \sum_{i=a}^b \mathbb{C}[x]F(x^{k^i}) \,.
$$
We note that $I$ is an ideal of $\mathbb C[x]$. Let $P_0(x)$ be a generator for $I$.  
Let us assume that $\alpha$ is a root of $P_0(x)$ with the property that $\alpha^{k^i}=\alpha$ 
for some positive integer $i$.  We will obtain a contradiction from this assumption. 

\medskip

Since $F(x)$ is $k$-Mahler, it is also $k^{i}$-Mahler and hence $F(x)$ satisfies a non-trivial polynomial equation
$$\sum_{j=0}^d Q_j(x) F(x^{k^{ij}}) \ = \ 0$$ with $Q_0,\ldots ,Q_d$ polynomials.  
We pick such a nontrivial relation with $Q_0$ nonzero and the degree of $Q_0$ minimal. 
By assumption $P_0$ divides $Q_0$ and so $\alpha$ is a root is of $Q_0(x)$.  
Also, we may assume that for some integer $j$, $0<j\leq d$, we have $Q_j(\alpha)\neq 0$. Indeed, otherwise 
we could divide our equation by $(x-\alpha)$ to get a new relation with a new $Q_0$ of smaller degree.

\medskip

By Lemma \ref{lem: gettingridofzeros}, there exists some natural number $N$ such that $F(x)$ can be decomposed as 
$F(x)=T(x)+x^NF_0(x)$, where $T(x)$ is a polynomial of degree $N-1$ and $F_0(x)$ is a power series with nonzero 
constant term such that 
$F_0(x)$ satisfies a $k^i$-Mahler equation
\begin{equation}
\label{eq: tilde}
\sum_{j=0}^e \widetilde{Q}_j(x) F_0(x^{k^{ij}}) \ = \ 0
\end{equation}
 with $\widetilde{Q}_0(0)=1$, $\widetilde{Q}_0(\alpha)=0$ and $\widetilde{Q}_j(\alpha)\neq 0$ for some integer $j$, $0<j\leq e$.  
 Moreover, by picking $N$ sufficiently large, we may assume that $F_0(x)$ satisfies a 
 nontrivial $\ell$-Mahler equation 
 $$
 \sum_{j=0}^f R_j(x) F_0(x^{\ell^{j}}) \ = \ 0
 $$ 
for some polynomials $R_j(x)$ with $R_0(0)=1$.   
Now, we infer from Proposition \ref{rem: decomp} that there is some $\ell$-Becker power series 
$G(x)$ such that 
\begin{equation}\label{eq: M}
F_0(x)= \left(\prod_{j=0}^{\infty} R_0(x^{\ell^j})\right)^{-1}G(x) \, .
\end{equation}
 
For $i=0,\ldots ,e$, we let $c_i$ denote the order of vanishing of $\widetilde{Q}_i(x)$ at $\alpha$, with the convention that 
$c_i=\infty$ if $\widetilde{Q}_i(x)=0$.  We note that $0<c_0<\infty$ and that there is some $j$, $0<j\leq e$, 
such that $c_j=0<c_0$. 
Let 
\begin{equation}\label{eq: defb}
b:= \max \left\{\frac{c_0-c_j}{j}  \mid j=1,\ldots ,d\right\}\, .
\end{equation}
Since at least one of $c_1,\ldots ,c_d$ is strictly less than $c_0$, we have that $b$ is positive.  
 Moreover, by definition there is some $j_0\in\{1,\ldots,d\}$ such that $c_{j_0}+bj_{0}-c_0=0$. 
Then, for $j\in \{0,\ldots ,d\}$, we set 
\begin{equation}\label{eq: Sj} 
S_j(x) := \widetilde{Q}_j(x)\left(\prod_{n=0}^{j-1} (1-\alpha^{-1} x^{k^{in}})^b \right) (1-\alpha^{-1} x)^{-c_0} \, . 
\end{equation}
Note that (\ref{eq: defb}) implies that $S_0(x)$ is 
a polynomial in $\mathbb C[x]$ such that $S_0(0)=1$ and $S_0(\alpha)\not=0$. 

\medskip
 
Now, we set   
\begin{equation}
\label{eq: Ldef}
L(x) := F_0(x)  \prod_{j=0}^{\infty} S_0(x^{k^{ij}}) \prod_{j=0}^{\infty} (1-\alpha^{-1} x^{k^{ij}})^b 
\end{equation}  
and we infer from (\ref{eq: M}) that 
\begin{equation}
\label{eq: Bigone} L(x)\prod_{j=0}^{\infty} (1-\alpha^{-1} x^{k^{ij}})^{-b} \prod_{j= 0}^{\infty} R_0(x^{\ell^j}) \ = \ 
G(x) \prod_{j=0}^{\infty} S_0(x^{k^{ij}}) \, .
\end{equation}  
In the sequel, we are going to obtain some asymptotic estimates for the quantities $G(x)$, $\displaystyle\prod_{j\ge 0} R_0(x^{\ell^j})$, 
$\displaystyle\prod_{j\geq 0} S_0(x^{k^{ij}})$, $L(x)$ and $\displaystyle\prod_{j\geq 0} (1-\alpha^{-1} x^{k^{ij}})^{-b}$ in 
a neighbourhood of some root of unity. 
We will then show that these estimates are not compatible with Equality (\ref{eq: Bigone}), providing the desired contradiction.

\medskip

\subsubsection*{ Estimate for $G(x)$} 
Note first that, since $G(x)$ is a $\ell$-Becker power series, Theorem \ref{thm: beck} implies that $G(x)$ is 
$\ell$-regular. By Proposition \ref{thm: analyticunitdisc}, there exist two positive real numbers $C$ and $m$ such that 
$$
|G(x)| <   C(1-|x|)^{-m} \, ,
$$
for even complex number $x$ in the open unit disc. 
This implies that there exist two positive real numbers $A_0$ and $\varepsilon_0$ such that
\begin{equation}
\label{eq: MM}
|G(x)| <   (1-|x|)^{-A_0}
\end{equation}
 for every complex number $x$ with $1-\varepsilon_0< 1-\vert x \vert < 1$.

\medskip

\subsubsection*{ Asymptotic estimate for  $\prod_{j\ge 0} R_0(x^{\ell^j})$} 
By assumption there is a prime $p$ that divides $k$ and does not divide $\ell$. Thus 
there exists some positive integer $N_0$ such that whenever $\zeta$ is a primitive $p^n$-th root 
of unity with $n\ge N_0$, then we have $R_0((\alpha\zeta)^{\ell^j})$ is nonzero for every nonnegative integer $j$.

Let $\zeta$ be such a primitive $p^n$-th root of unity with $n\ge N_0$.  Then
there exist two positive integers $n_1$ and $n_2$, $n_1<n_2$, such that
\begin{equation}\label{eq: n2n1}
(\alpha\zeta)^{\ell^{n_1}}=(\alpha\zeta)^{\ell^{n_2}} \, .
\end{equation}
Then for $t\in (0,1)$ we have
$$
\prod_{j= 0}^{\infty} R_0((t\alpha\zeta)^{\ell^j}) = \prod_{j=0}^{n_1-1}R_0((t\alpha\zeta)^{\ell^j}) 
\prod_{i=n_1}^{n_2-1} \prod_{j= 0}^{\infty} R_0(((t\alpha\zeta)^{\ell^{i}})^{\ell^{j(n_2-n_1)}}) \, .
$$
Note that $\prod_{j=0}^{n_1-1}R_0(x)$ 
is a polynomial that does not vanish at any point of the finite set $\left\{(\alpha\zeta)^{\ell^j}) \mid j\geq 0\right\}$.  
This gives that there exist two positive real numbers $\delta$ and $\varepsilon_1$ such that 
$$
\left\vert \prod_{j=0}^{n_1-1}R_0(t\alpha\zeta)^{\ell^j})\right\vert > \delta \, ,
$$
for all $t\in (1-\varepsilon_1,1)$. 
Furthermore, Equality (\ref{eq: n2n1}) implies that for every integer $i$, $n_1\leq i \leq n_2-1$, we have
$$
((\alpha\zeta)^{\ell^{i}})^{\ell^{j(n_2-n_1)}} = ((\alpha\zeta)^{\ell^{i}}) \, .
$$ 
Thus, for every integer $i$, $n_1\leq i \leq n_2-1$, we can apply Corollary \ref{cor: mog} to the infinite product 
$$
 \prod_{j= 0}^{\infty} R_0(((t\alpha\zeta)^{\ell^{i}})^{\ell^{j(n_2-n_1)}}) \, .
$$
This finally implies the existence of 
a positive real number $\varepsilon_2=\varepsilon_2(\zeta)$ and a positive integer $A_1= A_1(\zeta)$ 
such that
\begin{equation}
\label{eq: SS}
\left| \prod_{j= 0}^{\infty} R_0((t\alpha\zeta)^{\ell^j})\right|> (1-t)^{A_1}
\end{equation}
for $t\in (1-\varepsilon_2,1)$.

\medskip

\subsubsection*{ Asymptotic estimate for $\prod_{j\ge 0} S_0(x^{k^j})$} 
First note that since $\alpha^k=\alpha$, $S_0(0)=1$ and $\alpha$ is not a root of $S_0$, 
we can apply Corollary \ref{cor: mog}. We thus obtain the existence of a positive real number $\delta_0$ and 
a positive integer $M_0$ such that  
\begin{equation}\label{eq: M0}
\left \vert \prod_{j= 0}^{\infty} S_0( (t\alpha)^{k^{ij}}) \right \vert < (1-t)^{M_0}
\end{equation}
for every $t\in (1-\delta_0,1)$. 

Now, if $\zeta$ is a primitive $p^n$-th root of unity, for some positive integer $n$, 
we have $(\alpha\zeta)^{k^{ij}}=\alpha$ for all $j\ge n$. This implies that
\begin{equation}\label{eq: RS}
\prod_{j= 0}^{\infty} S_0((t\alpha\zeta)^{k^{ij}}) = R(t)\prod_{j= 0}^{\infty} S_0((t\alpha)^{k^{ij}}) \, ,
\end{equation}
where 
$$
R(t)= \left(\prod_{j=0}^{n-1} S_0((t\alpha\zeta)^{k^{ij}}) \right) 
\left(\prod_{j=0}^{n-1} S_0((t\alpha)^{k^{ij}})\right)^{-1} \, .
$$  
Since $\alpha^{k^{ij}}=\alpha$ and $S_0(\alpha)\not=0$, there are two positive real number $\delta_1$ and $C$ such that 
\begin{equation}\label{eq: R}
\vert R(t)\vert < C
\end{equation} 
for every $t\in (1-\delta_1,1)$.  
We thus infer from (\ref{eq: M0}),  (\ref{eq: RS}) and  (\ref{eq: R}) that there exist a positive real number 
$\varepsilon_3=\varepsilon_3(\zeta)$ and a positive integer $A_2=A_2(\zeta)$ such that
\begin{equation}
\label{eq: TT}
\left| \prod_{j= 0}^{\infty} S_0((t\alpha\zeta)^{\ell^j})\right|< (1-t)^{-A_2}
\end{equation}
for $t\in (1-\varepsilon_3,1)$.

\medskip

\subsubsection*{ Asymptotic estimate for $L(x)$}

We first infer from (\ref{eq: tilde}) and (\ref{eq: Ldef}) that the function $L$ satisfies the following relation: 
$$
\sum_{n=0}^e \widetilde{Q}_n(x)  \left(\prod_{j= n}^{\infty}S_0(x^{k^{ij}})^{-1} \right)
\left(\prod_{j= n}^{\infty} (1-\alpha^{-1}x^{k^{ij}})^{-b}\right)  L(x^{k^{in}}) = 0 \, ,
$$
which gives by (\ref{eq: Sj} ) that 
\begin{eqnarray*}
L(x) &= &- \sum_{n=1}^{e} \left ( 
 \widetilde{Q}_n(x) \widetilde{Q}_0(x)^{-1}  \prod_{j=0}^{n-1} S_0(x^{k^{ij}}) 
\prod_{j = 0}^{n-1} (1-\alpha^{-1}x^{k^{ij}})^b \right)   L(x^{k^{in}}) \\ \\
&=&  \left (  \widetilde{Q}_n(x)  \left(\prod_{j= 0}^{n-1} (1-\alpha^{-1}x^{k^{ij}})^b\right)  (1-\alpha^{-1}x)^{-c_0} S_0(x)^{-1}  
\prod_{j=0}^{n-1} S_0(x^{k^{ij}})   \right)   L(x^{k^{in}})\\  \\
& = & - \sum_{n=1}^{e}  \left(S_n(x)  \prod_{j=1}^{n-1} S_0(x^{k^{ij}}) \right)
   L(x^{k^{in}}) \, .
\end{eqnarray*} 
Let $A(x)$ denote the $e\times e$ matrix whose $(i,j)$-entry is $\delta_{i,j+1}$ if $i\ge 2$ and whose 
$(1,j)$-entry is
$$
C_j(x) := -  S_n(x)  \prod_{j=1}^{n-1} S_0(x^{k^{ij}}) 
$$ 
for $j=1,\ldots ,e$.  
Then the previous computation gives us the following functional equation:
\begin{equation}
\label{eq: LL}
[L(x), L(x^{k^i}),\ldots , L(x^{k^{i(e-1)}})]^T \ = \ A(x)[L(x^{k^i}),\ldots , L(x^{k^{ie}})]^T \, .
\end{equation}

\medskip

We claim that if $\zeta$ is a primitive $p^n$-th root of unity with $n\ge N_0+i(e-1)\nu_p(k)$, 
then there exist a positive integer $M_0= M_0(\zeta)$ and an infinite sequence $(t_n)_{n\geq 0}\in (0,1)^{\mathbb N}$ 
which tends to $1$ such that 
\begin{equation}\label{eq: minL1}
|| \, [L(t_n\alpha\zeta), L(t_n^{k^i}\alpha\zeta^{k^i}),\ldots , L(t_n^{k^{i(e-1)}}\alpha\zeta^{k^{i(e-1)}})]^T
\, ||> (1-t_n)^{M_0} \, .
\end{equation}

\medskip

In order to obtain Inequality (\ref{eq: minL1}) it remains to prove that we can apply Corollary \ref{cor: kappa} to $L(x)$.  
Note that $L(x)$ is not identically zero since $F(x)$ is not a polynomial. Furthermore, 
we can assume that $L$ is not a nonzero constant since otherwise Inequality (\ref{eq: minL1}) would be immediately satisfied. 

\medskip

\begin{itemize}

\item[(i)] By definition,  $S_n(x) = \widetilde{Q}_n(x) \left(\prod_{j=0}^{n-1} (1-\alpha^{-1} x^{k^{ij}})^b \right) 
(1-\alpha^{-1} x)^{-c_0}$. Moreover, a simple computation gives that 
$\prod_{j=0}^{n-1} (1-\alpha^{-1} x^{k^{ij}})^b =  (1-\alpha^{-1} x)^{bn} P_n(x)^b$, 
for some polynomial $P_n(x)$ that does not vanish at $\alpha$.  By definition of $c_n$, this shows that
\begin{equation}\label{eq: pr}
S_n(x) = (1-\alpha^{-1} x)^{c_n+bn-c_0}P_n(x)^bR_n(x) \, ,
\end{equation} 
where $P_n(x)$ and $R_n(x)$ are two polynomials that do not vanish at $\alpha$. 
By definition of $b$, we have $c_n+bn-c_0\geq 0$ for $n\in\{0,\ldots,e\}$, and 
thus $S_n(x)$ is analytic in the open unit disc and continuous on the closed unit disc.  
Since the finite product $\prod_{j=1}^{n-1} S_0(x^{k^{ij}})$ 
is a polynomial, this shows that the entries of the matrix $A(x)$ are analytic on 
$B(0,1)$ and continuous on $\overline{B(0,1)}$.

\medskip

\item[(ii)]  The definition of $b$ implies that there is some integer $r$, 
$1\leq r\leq e$, such that $c_r+br -c_0=0$. Since $P_r(\alpha)R_r(\alpha)\not=0$, 
Equation (\ref{eq: pr}) implies that $S_r(\alpha)\not=0$. On the other hand, we have that $\prod_{j=0}^{r-1} S_0(x^{k^{ij}})$ 
does not vanish at $\alpha$ since $S_0(\alpha)\not=0$ and $\alpha^{k^i}=\alpha$. 
We thus obtain that the 
$(1,r)$-entry of $A(\alpha)$ is nonzero.  
By Lemma \ref{lem: nilp}, this implies that $A(\alpha)$ is not nilpotent.

\medskip

\item[(iii)] By definition of the matrix $A$, we get that 
$$
\det A(x) = (-1)^{e} C_e(x)=  (-1)^{e+1} S_e(x) \prod_{n=1}^{e-1} S_0(x^{k^{in}}) \, .
$$
By (\ref{eq: pr}), we have that $S_e(x)= (1-\alpha^{-1} x)^{c_e+be-c_0} P_e(x)^bR_e(x)$, where $P_e(x)$ and $R_e(x)$ 
are polynomials.  
It follows that 
there exist two positive real numbers $\delta$ and $M$ such that
\begin{equation*}
|\det A(x)| >  (1-|x|)^{M}
\end{equation*}
for every $x$ such that $1-\delta < \vert x\vert <1$. 

\medskip

\item[(iv)]   We claim that 
$$
\left\{[L(x), L(x^{k^i}),\ldots , L(x^{k^{i(e-1)}})]^T~\mid~x\in B(0,1)\right\}
$$ 
cannot be contained in a proper subspace of $\mathbb{C}^e$.  Indeed, if it were, then there would exist some nonzero row 
vector $u$ such that
$$
u[L(x), L(x^{k^i}),\ldots , L(x^{k^{i(e-1)}})]^T \ = \ 0
$$ for all $x\in B(0,1)$.  But this would give that
$L(x),\ldots ,L(x^{k^{i(e-1)}})$ are linearly dependent over $\mathbb{C}$, and hence by Lemma \ref{lem: constant}, we would obtain 
that $L(x)$ is a constant function, a contradiction.  

\end{itemize}

\medskip

It follows from (i), (ii), (iii) and (iv) that we can apply Corollary \ref{cor: kappa} to $L(x)$, which proves that (\ref{eq: minL1}) holds. 
Then, we deduce from  (\ref{eq: minL1}) that there exist a sequence $(s_n)_{n\geq 0}$ in $(0,1)$ which tends to $1$, 
some root of unity $\mu$ that has order at least $p^{N_0}$ and some positive integer $A_3=A_3(\zeta)$  such that
\begin{equation} \label{eq: LLL}
|L(s_n \alpha\mu)|> (1-s_n)^{A_3}
\end{equation} for every positive integer $n$.

\medskip

\subsubsection*{ Conclusion} 
By Equation (\ref{eq: Bigone}), we have 
\begin{eqnarray*} 
& & \left|L(s_n\alpha\mu)\prod_{j=0}^{\infty} (1-\alpha^{-1} (s_n\mu)^{k^{ij}})^{-b} \prod_{j= 0}^{\infty} 
 R_0((s_n \alpha\mu )^{\ell^j})\right| \\ 
& = &   
\left|G(s_n\alpha\mu ) \prod_{j=0}^{\infty} S_0(\alpha (s_n\mu)^{k^{ij}})\right|.
\end{eqnarray*}
By Equations (\ref{eq: MM}) and (\ref{eq: TT}), we see that the right-hand side is at most
$$
(1-s_n)^{-(A_0+A_2)}
$$ 
for every integer $n$ large enough.
Similarly, by Equations (\ref{eq: SS}) and (\ref{eq: LLL}), the left-hand side is at least
$$
 (1-s_n)^{A_1+A_3}\prod_{j=0}^{\infty} (1-\alpha^{-1} (s_n\mu)^{k^{ij}})^{-b}
 $$
for every integer $n$ large enough. 
Thus we have 
$$
\left|\prod_{j=0}^{\infty} (1-\alpha^{-1} (s_n\mu)^{k^{ij}})^{-b}\right|
< (1-s_n)^{-(A_0+A_1+A_2+A_3)}
$$
for every integer $n$ large enough.  
But this contradicts Lemma \ref{lem: last}, 
since $\mu^{k^j}=1$ for all sufficiently large $j$.  
This concludes the proof. 
\end{proof}



\section{Existence of prime ideals with special properties}\label{sec: chebotarev}

In this section we prove the following result.
  
\begin{thm} Let $R$ be a principal localization of a number ring and 
let $P(x),Q(x)\in R[x]$ be two polynomials with $P(0)=Q(0)=1$ and such that none of the zeros of $P(x)Q(x)$ 
are roots of unity.  Then there are infinitely many prime ideals $\mathfrak P$ in $R$ such that   
  $$
  \left(\prod_{i=0}^{\infty} P(x^{k^i}) \right)^{-1} \bmod \mathfrak P
  $$
is a $k$-automatic power series in $(R/\mathfrak P)[[x]]$ and 
$$
 \left( \prod_{i=0}^{\infty} Q(x^{{\ell}^i})\right)^{-1} \bmod \mathfrak P 
  $$ 
  is a $\ell$-automatic power series in $(R/\mathfrak P)[[x]]$. 
\label{thm: reduction}
\end{thm}

Our proof is based on Chebotarev's density theorem for which we refer the reader for example 
to \cite{Lang} and to the informative survey 
\cite{LeSt}.  We first prove a basic lemma about non-existence of 
$n$-th roots of elements in a number field for sufficiently large $n$.  
The proof makes use of the notion of Weil absolute logarithmic height.  
We do not recall the precise definition of Weil height, as it is a bit long and not really within the scope of the present paper.  
However, we are only going to use basic properties of this height that can be found in any standard book such as \cite{HiSi}, \cite{LangH},  
or \cite{Wald}.

 \begin{lem} Let $K$ be a number field and let $\alpha$ be a nonzero element in $K$ that is not a root of unity.   
 Then for all sufficiently large natural numbers $n$ the equation $ \beta^n=\alpha$  has no solution $\beta\in K$. 
 \label{lem: power}
\end{lem}

\begin{proof} 
This result is an easy consequence of the theory of heights.  
Given $x\in K$, we denote by $h(x)$ the Weil absolute logarithmic height of $x$. 

\medskip

Since $K$ is a number field, it has the Northcott property, that is for every positive 
real number $M$ the set $\{x\in K \mid h(x) \leq M\}$ is finite. In particular, there exists a positive real number $\varepsilon$ depending 
only on $K$ such that if $h(x) < \varepsilon$ then $h(x)=0$. Let $n$ be an integer such that $n>h(\alpha)/\varepsilon$. 
Let us assume that there is  $\beta\in K$ such that $\beta^n=\alpha$.  Since $h(x^k)=kh(x)$ for every $x\in K$ and $k\in \mathbb N$, 
we obtain that $h(\beta) = h(\alpha)/n<\varepsilon$.  
Thus $h(\beta)=0$. By Kronecker's theorem, this implies that $\beta$ is a root of unity and thus $\alpha$ is also be a root of unity, 
a contradiction.  
\end{proof}

\begin{lem}\label{lem: group}
Let $m$ be a natural number and let $d_1,\ldots ,d_m$ be positive integers.   
Suppose that $H$ is a subgroup of $$\prod_{i=1}^m \left(\mathbb{Z}/d_i\mathbb{Z}\right)$$ 
with the property that there exist natural numbers $r_1,\ldots ,r_m$ with $$1/r_1+\cdots +1/r_m<1$$ 
such that for each $i\in \{1,\ldots ,m\}$, there is an element $h_i\in H$ whose $i$-th coordinate has order $r_i$.
Then there is an element $h\in H$ such that no coordinate of $h$ is equal to zero.
\end{lem}

\begin{proof} 
For each $i\in \{1,\ldots ,m\}$, we let 
$$\pi_i :\prod_{i=1}^m \left(\mathbb{Z}/d_i\mathbb{Z}\right) \to \mathbb{Z}/d_i\mathbb{Z}$$ 
denote the projection onto the $i$-th coordinate.
Given $(x_1,\ldots ,x_m)\in \mathbb{Z}^m$
we have that $x_1h_1+\cdots +x_m h_m\in H$.  
Observe that the density of integers $y$ for which
$$
\pi_i\left(\sum_{j\not =i} x_j h_j + y h_i\right)=0
$$ 
is equal to $1/r_i$.  Since this holds for all
$(x_1,x_2,\ldots ,x_{i-1},x_{i+1},\ldots ,x_m)\in  \mathbb{Z}^{m-1}$, we see that the density of
$(x_1,\ldots ,x_m)\in \mathbb{Z}^m$ for which
$$
\pi_i\left(\sum_{j=1}^m x_j h_j \right)=0
$$ 
is equal to $1/r_i$.  Thus the density of 
$(x_1,\ldots ,x_m)\in \mathbb{Z}^m$ for which
$$
\pi_i\left(\sum_{j=1}^m x_j h_j \right)=0
$$ 
holds for some $i\in\{1,\ldots,m\}$ is at most
$$
1/r_1+\cdots +1/r_m<1 \, .
$$  
In particular, we see that there is some
 $(x_1,\ldots ,x_m)\in \mathbb{Z}^m$ such that 
the element $h:=x_1h_1+\cdots +x_m h_m\in H$ has no coordinate equal to zero.
\end{proof}

\begin{lem} Let $k\geq 2$ be an integer, let $R$ be a principal localization of a number ring, 
let $\mathfrak P$ be a nonzero prime ideal of $R$, and let $a$ be an element of $R$.  Suppose that for some natural number $n$, 
the polynomial $1-ax^{k^n} \bmod \mathfrak P$ has no roots in $R/\mathfrak P$.   
Then the infinite product 
$$
 \left(\prod_{j=0}^{\infty} (1-ax^{k^j})\right)^{-1} \bmod \mathfrak P
$$ 
is a $k$-automatic power series in $(R/\mathfrak P) [[x]]$.
\label{lem: auto}
\end{lem}

\begin{proof}  Set $F(x):= \displaystyle\prod_{j=0}^{\infty} (1-ax^{k^j})^{-1} \bmod \mathfrak P$. 
Without loss of generality we can assume that $a$ does not belong to $\mathfrak P$. 
Let us first note that the sequence
$a,a^k,a^{k^2},\ldots $ is necessarily eventually periodic modulo $\mathfrak P$. 
However, it cannot be periodic, as otherwise the polynomial 
$1-ax^{k^n}$ would have a root for every natural number $n$.  
Thus there exists  a positive integer $N$ such that 
$$
a\not\equiv a^{k^{N}}\equiv a^{k^{2N}}  ~\bmod~\mathfrak P \,.
$$
Set $b:=a^{k^N}$ and let us consider the polynomial
$$
Q(x):=(1-bx)(1-bx^k)\cdots (1-bx^{k^{N-1}}) \, .
$$
Now arguing exactly as in the proof of Proposition \ref{prop: omega}, we see that 
there exist polynomial $S(x)\in R[x]$ such that $G(x):=Q(x)^{-1}F(x)$ satisfies the equation 
$$
G(x)\equiv S(x)G(x^k) ~ \bmod \mathfrak P \, .
$$
Thus Theorem 5.3 implies that $G(x)\bmod \mathfrak P$ is a $k$-regular power series in $(R/\mathfrak P)[[x]]$.  
By Proposition \ref{prop: reg2},  we see that $F(x)\bmod \mathfrak P$ is  a $k$-regular power series since it is  a 
product of a polynomial (which is $k$-regular) and a $k$-regular power series.  
Since the base field is finite, Proposition \ref{prop: reg2} gives that $F(x)\bmod \mathfrak P$ is actually 
a $k$-automatic power series.  
This ends the proof.  
\end{proof}

\begin{proof}[Proof of Theorem \ref{thm: reduction}] 
By assumption $R$ is a principal localization of a number field $K$. 
Let $L$ be the Galois extension of $K$ generated by all complex roots of the polynomial $P(x)Q(x)$. 
Thus there are $\alpha_1,\ldots,\alpha_d,\beta_1,\ldots,\beta_e\in L$ such that 
$P(x)=(1-\alpha_1x)\cdots (1-\alpha_dx)$ and $Q(x)=(1-\beta_1x)\cdots (1-\beta_ex)$.  
We fix a prime $p$ that divides $k$ and a prime $q$ that divides $\ell$.
Let $s$ be a natural number such that $p^s$ and $q^s$ are both larger than $d+e$.  
Since by assumption none of the roots of $P(x)Q(x)$ is a root of unity, 
Lemma \ref{lem: power} implies that, for $1\le i\le d$ and $1\le j\le e$, there 
are largest nonnegative integers $n_i$ and $m_j$ with the property that we can write $\alpha_i=\gamma_i^{p^{n_i}}u_i$ 
and $\beta_j=\delta_j^{q^{m_j}}v_j$ for some elements $\gamma_i,\delta_j\in L(e^{2\pi i/(p^s q^s)})$ and $u_i,v_j$ 
roots of unity in $L(e^{2\pi i/(p^s q^s)})$.

  \medskip
  
 Next let $n$ denote a natural number that is strictly larger than the maximum of the $n_i$ 
 and the $m_j$ for $1\le i\le d$ and $1\le j\le e$. 
Set $E:=L(e^{2\pi i/(p^n q^n)})$ and let $F$ denote the Galois extension of $E$ generated by all complex 
roots of the polynomial 
$$
\prod_{i=1}^d \prod_{j=1}^{e} (x^{p^n}-\gamma_i)(x^{q^n}-\delta_j) \,.
$$  
For each $i$, $1\leq i \leq d$, we pick a root 
$\gamma_{i,0}$ of $x^{p^n}-\gamma_i$, and for each $j$, $1\leq j\leq e$, 
we pick a root $\delta_{j,0}$ of $x^{q^n}-\delta_j$.

\medskip

\noindent{\it Claim.} We claim that for every integer $i$, $1\leq i\leq d$, 
there is an automorphism $\sigma_i$ in ${\rm Gal}(F/E)$ such that 
$$
\sigma_i(\gamma_{i,0}) = \gamma_{i,0} u \, ,
$$ 
with $u$ a primitive $p^r$-th root of unity for some $r$ 
greater than or equal to $s$.  
Similarly,  for every integer $j$, $1\leq j\leq e$,  there is an automorphism 
$\tau_j$ in ${\rm Gal}(F/E)$ that such that
$$
\tau_j(\delta_{j,0})=\gamma_{j,0} u' \,,
$$ 
for some primitive $q^{r'}$-th root of unity $u'$ with $r'$ greater than or equal to $s$.

\medskip

\noindent{\it Proof of the claim.} 
Note that 
$$
\left\{\frac{\sigma(\gamma_{i,0})}{\gamma_{i,0}}~\mid~\sigma\in {\rm Gal}(F/E)\right\}
$$ 
forms a subgroup of the $p^n$-th roots of unity.  
To prove the claim we just have to prove that this group cannot be contained in the group of $p^{s-1}$-st roots of unity. 
Let us assume that this is the case. Then the product of the Galois conjugates 
of $\gamma_{i,0}$ must be $\tilde{\gamma_i}:=\gamma_{i,0}^{p^t}v$ for some $t< s$ 
and some root of unity $v$.  Moreover, $\tilde{\gamma_i}$ lies in $L(e^{2\pi i/(p^nq^n)})$.  
Note that the Galois group of $L(e^{2\pi i/(p^nq^n)})$ over $L(e^{2\pi i/(p^s q^s)})$ has order dividing 
$\phi(p^nq^n)/\phi(p^s q^s)= p^{n-s}q^{n-s}$.  Since all conjugates of $\tilde{\gamma_i}$ are equal 
to $\tilde{\gamma_i}$ times some root of unity, we see that the relative norm of 
$\tilde{\gamma_i}$ with respect to the subfield $L(e^{2\pi i/(p^s q^s)})$ is of the form
$\tilde{\gamma_i}^d v'$ for some divisor $d$ of $p^{n-s}q^{n-s}$ and some root of unity
 $v'$.  Moreover, $$\tilde{\gamma_i}^d v' \in L(e^{2\pi i/(p^s q^s)}) \, .$$  
 Note that the gcd of $d$ and $p^{n-t}$ is equal to $p^{n-s_0}$ for some integer $s_0\geq s$. 
 Since $\gamma_{i,0}^{p^n}= \tilde{\gamma_i}^{p^{n-t}} v^{-p^{n-t}}\in L(e^{2\pi i/(p^s q^s)})$, 
 we see by expressing $p^{n-s_0}$ as an integer linear combination of  $d$ and $p^{n-t}$ that
 $$
 \tilde{\gamma_i}^{p^{n-s_0}}\omega = \gamma_{i,0}^{p^{n-s_0+t}}\omega' 
 \in L(e^{2\pi i/(p^s q^s)})
 $$ 
 for some roots of unity $\omega$ and $\omega'$ and some $s_0\ge s$.  
 But $s_0-t\ge 1$ and so we see 
 that $\alpha_i$ is equal to a root of unity times 
$$
\left(\gamma_{i,0}^{p^{n-s_0+t}}\omega'\right)^{p^{s_0-t+n_i}} \,,
$$
contradicting the maximality of $n_i$.   This confirms the claim. $\square$

\medskip

For an integer $m$, we let $\mathbb U_m$ denote the subgroup of $\mathbb{C}^*$ 
consisting of all $m$-th roots of unity.  
Note that we can define a group homomorhpism $\Phi$ from ${\rm Gal}(F/E)$ to 
$(\mathbb U_{p^n})^d\times (\mathbb U_{q^n})^e$ by 
$$
\Phi(\sigma) :=  (\sigma(\gamma_{1,0})/\gamma_{1,0}, \ldots, \sigma(\gamma_{d,0})/\gamma_{d,0}, 
\sigma(\delta_{1,0})/\delta_{1,0},\ldots,\sigma(\delta_{e,0})/\delta_{e,0})  \, .
$$ 
We see that $\Phi$ is a group homomorphism since each $\sigma\in {\rm Gal}(F/E)$ fixes the $p^n$-th and $q^n$-th roots of unity.  
Set $H:=\Phi({\rm Gal}(F/E))$. 
The claim implies that the $i$-th coordinate in $(\mathbb U_{p^n})^d$ of $\Phi(\sigma_i)$  has order at least equal to $p^s$. 
Similarly, it also implies that the $j$-th coordinate  in $(\mathbb U_{q^n})^e$ of $\Phi(\tau_j)$ has order at least equal to $q^s$.  
Since $p^s$ and $q^s$ are both greater than $d+e$, we have
$$
d/p^s+e/q^s < 1\, .
$$ 
Now, since $(\mathbb U_{p^n})^d\times (\mathbb U_{q^n})^e\cong 
(\mathbb Z/p^n\mathbb Z)^d\times (\mathbb Z/q^n \mathbb Z)^e$, we infer from Lemma \ref{lem: group} 
that there exists an element $h$ in $H$ such that every coordinate of $h$ is different from the identity element. 
In other words, this means that there exists some element $\tau$ of ${\rm Gal}(F/E)$ that fixes no element in the set  
$$
\{\gamma_{i,0} \mid 1\leq i \leq d\} \cup \{\delta_{j,0} \mid 1 \leq j \leq e\} \, .
$$ 
Since by definition $\tau$ fixes all $p^n$-th and $q^n$-th roots of unity, 
we see more generally that no root of the polynomial 
$$
\prod_{i=1}^d \prod_{j=1}^{e} (x^{p^n}-\gamma_i)(x^{q^n}-\delta_j) 
$$
is fixed by $\tau$.  
Since $\tau$ belongs to ${\rm Gal}(F/E)$, we can see $\tau$ as an element of ${\rm Gal}(F/K)$ that fixes all elements of 
$E$. We have thus produce an element $\tau$ of ${\rm Gal}(F/K)$ that fixes all roots of $P(x)Q(x)$ but that 
that does not fix any of the roots of the polynomial 
$$
\prod_{i=1}^d \prod_{j=1}^{e} (x^{p^n}-\gamma_i)(x^{q^n}-\delta_j) \, .
$$ 
It follows from Chebotarev's density theorem (see for instance the discussion in \cite{LeSt}) 
that there is an infinite set of nonzero prime ideals $\mathcal S\subseteq {\rm Spec}(R)$  
such that if $\mathfrak P\in \mathcal S$ then $P(x)Q(x) \bmod \mathfrak P$ factors into linear terms while  
the minimal polynomial of 
$$
\prod_{i=1}^d \prod_{j=1}^{e} (x^{p^n}-\gamma_i)(x^{q^n}-\delta_j)
$$ 
over $K$ has no root modulo $\mathfrak P$.  
In particular, there is a natural number $N$ larger than $n$ such that for all such 
prime ideals $\mathfrak P$, the polynomial $P(x)Q(x)\bmod\mathfrak P$ splits into linear factors,  
while the polynomial $P(x^{p^N})Q(x^{q^N}) \mod \mathfrak P$ does not have any roots in $R/\mathfrak P$.  

\medskip

For such a prime ideal $\mathfrak P$, there thus exist $a_1,\ldots,a_d,b_1,\ldots,b_e$ 
in the finite field $R/\mathfrak P$ such that 
$$
P(x)\equiv (1-a_1x)\cdots (1-a_dx)~ \bmod\mathfrak P
$$ 
and 
$$
Q(x)\equiv (1-b_1x)\cdots (1-b_dx)~ \bmod\mathfrak P \, .
$$ 
Then
$$
\left(\prod_{j= 0}^{\infty} P(x^{k^j})\right)^{-1} \equiv \prod_{i=1}^d\left( \prod_{j= 0}^{\infty} 
(1-a_ix^{k^j})\right)^{-1} \bmod \mathfrak P\, .
$$
By Lemma \ref{lem: auto} the right-hand side is a product of $k$-automatic power series and hence, 
by Proposition \ref{prop: reg2},  is $k$-automatic.  
Thus the infinite product 
$$
\left(\prod_{j= 0}^{\infty} P(x^{k^j})\right)^{-1}\bmod \mathfrak P
$$ 
is a $k$-automatic power series in $R/\mathfrak P[[x]]$. Similarly, we get that 
$$
\left(\prod_{j= 0}^{\infty}  Q(x^{\ell^j})\right)^{-1} \equiv \prod_{i=1}^e \left(\prod_{j=0 }^{\infty} (1-b_ix^{\ell^j})\right)^{-1} \bmod \mathfrak P \, ,
$$ 
which implies that the infinite product 
$$
\left(\prod_{j= 0}^{\infty} Q(x^{\ell^j})\right)^{-1}\bmod \mathfrak P
$$ 
is a $\ell$ automatic power series in $R/\mathfrak P[[x]]$.  
This concludes the proof. 
\end{proof}



\section{Proof of Theorem \ref{thm: main}}\label{sec: main}

We are now ready to prove our main result.

\begin{proof}[Proof of Theorem \ref{thm: main}]

Let $K$ be a field of characteristic zero and $k$ and $l$ be two multiplicatively independent positive integers. 

\medskip

We first note that if 
$F(x)\in K[[x]]$ is a rational function, then for every integer $m\geq 2$, it obviously satisfies a functional equation as in 
(\ref{eq: mahler}) with $n=0$. Hence, $F(x)$ is $m$-Mahler, which gives a first implication.   

\medskip

To prove the converse implication, we fix $F(x)\in K[[x]]$ that is both $k$- and $\ell$-Mahler and we aim at proving 
that $F(x)$ is a rational function. 
Of course, if $F(x)$ is a polynomial, there is nothing to prove.  
From now on, we thus assume that $F(x)$ is not a polynomial.
By Corollary \ref{cor: pqkl}, we can assume that there are primes $p$ and $q$ such that 
$p$ divides $k$ while $p$ does not divide $q$ and such that $q$ divides $\ell$ while $q$ does not divide $k$.  
By Theorem \ref{thm: red1}, we can assume that there is a ring $R$ that is a principal localization of a number ring such that 
$F(x)\in R[[x]]$ and satisfies the equations
$$
\sum_{i=0}^n P_i(x)F(x^{k^i}) \ = \ 0
$$ 
with $P_0,\ldots ,P_d\in R[x]$ 
and  
$$
\sum_{i=0}^m Q_i(x)F(x^{\ell^i}) \ = \ 0
$$ 
with $Q_0,\ldots ,Q_e\in R[x]$. 
Without loss of generality, we can assume that all complex roots of $P_0(x)$ and $Q_0(x)$ belong to $R$ 
(otherwise we could just enlarge $R$ by adding these numbers).  
Furthermore, we can assume that $P_0(x)Q_0(x)\neq 0$. 
By Corollary \ref{cor: P01Q01}, we can also
assume that $P_0(0)=1$ and that $Q_0(0)=1$, for otherwise we could just replace $F(x)$ by the power series 
$F_0(x)$ given there.   
We choose a ring embedding of $R$ in $\mathbb{C}$ and for the moment we regard 
$F(x)$ as a complex power series. 
By Theorem \ref{thm: elim}, we can assume that if $\alpha$ is a root of unity such that 
$\alpha^{k^j}=\alpha$ for some positive integer $j$, then $P_0(\alpha)\neq 0$. 
 Similarly, we can assume that if $\beta$ is a root of unity such that $\beta^{\ell^j}=\beta$ for some positive integer $j$, 
 then $Q_0(\beta)\neq 0$. 

\medskip

By Proposition \ref{rem: decomp}, we can write 
$$
F(x)=\left(\prod_{j=0}^{\infty} P_0(x^{k^j})\right)^{-1} G(x)\, ,
$$ 
for some $k$-regular power series $G(x)\in R[[x]]$. 
Furthermore, we can decompose $P_0(x)$ as $P_0(x)=S_0(x)S_1(x)$, where $S_0(x)$ and $S_1(x)$ are two polynomials, 
the zeros of $S_0(x)$ are all roots of unity, none of the zeros of $S_1(x)$ are roots of unity, and $S_0(0)=S_1(0)=1$.  
Since by assumption all roots of $P_0(x)$ lie in $R$, we get that both $S_0(x)$ and $S_1(x)$ belong to $R[x]$. 
By assumption if $\alpha$ is a root of $S_0(x)$ then for every positive integer $j$, one has $\alpha^{k^j}\neq \alpha$.  
Then, it follows from Proposition \ref{prop: omega} that
$$
\left(\prod_{j= 0}^{\infty} S_0(x^{k^j})\right)^{-1} \in R[[x]] 
$$ 
is a $k$-regular power series.  
Set $H:= \displaystyle\prod_{j= 0}^{\infty} S_0(x^{k^j})^{-1}G(x)$. We infer from Proposition \ref{prop: reg2} 
that $H(x)$ is a $k$-regular power series.  Moreover, one has 
\begin{equation}
\label{eq: decH}
F(x) = \left(\prod_{j = 0}^{\infty} S_1(x^{k^j})\right)^{-1} H(x) \, .
\end{equation}

Similarly, by Proposition \ref{rem: decomp}, we can write 
$$
F(x)=\left(\prod_{j = 0}^{\infty} Q_0(x^{k^j})\right)^{-1} I(x)\, ,
$$ 
for some $k$-regular power series $I(x)\in R[[x]]$. 
As previously, we can decompose $Q_0(x)$ as $Q_0(x)=T_0(x)T_1(x)$, 
where $T_0(x)$ and $T_1(x)$ belong to $R[x]$, 
the zeros of $T_0(x)$ are all roots of unity, none of the zeros of $T_1(x)$ are roots of unity, and $T_0(0)=T_1(0)=1$. 
 By assumption if $\beta$ is a root of $T_0(x)$ then for every positive integer $j$, one has $\beta^{\ell^j}\neq \beta$.  
Then it follows from Proposition \ref{prop: omega} that
$$
\left(\prod_{j= 0}^{\infty} T_0(x^{\ell^j})\right)^{-1} \in R[[x]]
$$ 
is a $\ell$-regular power series. 
Set $J:= \displaystyle\prod_{j= 0}^{\infty} T_0(x^{k^j})^{-1} I(x)$. Again, we see by Proposition \ref{prop: reg2} 
that $J(x)$  is $\ell$-regular.   
Moreover, one has 
\begin{equation}
\label{eq: decH2}
F(x) = \left(\prod_{j= 0}^{\infty} T_1(x^{k^j})\right)^{-1} J(x) \, .
\end{equation}

\medskip

By Theorem \ref{thm: reduction}, there is an infinite set of nonzero prime ideals $\mathcal{S}$ of $R$ such that,  
for every prime ideal $\mathfrak P$ in $\mathcal S$, 
$$
\left(\prod_{j= 0}^{\infty} S_1(x^{k^j})\right)^{-1} \bmod \mathfrak P
$$ 
is a $k$-automatic  power series in $(R/\mathfrak P)[[x]]$ 
and 
$$
\left(\prod_{j= 0}^{\infty} T_1(x^{\ell^j})\right)^{-1} \bmod \mathfrak P
$$ 
is a $\ell$-automatic power series in $(R/\mathfrak P)[[x]]$.   
Then we infer from Equalities (\ref{eq: decH}) and (\ref{eq: decH2}) that, 
for $\mathfrak P\in \mathcal{S}$,  $F(x) \bmod \mathfrak P$ is $k$-regular for it is  the product of two $k$-regular power series.   
Similarly, $F(x) \bmod \mathfrak P$ is a $\ell$-regular power series.   

We recall that the principal localization of a number ring is a  
\emph{Dedekind domain}; that is, it is a noetherian normal domain of  
Krull dimension one.  In particular, all nonzero prime ideals are  
maximal. Now since $R$ is a finitely generated $\mathbb Z$-algebra and $\mathfrak P$ is a maximal ideal, 
the quotient ring $R/\mathfrak P$ is a finite field (see \cite[Theorem 4.19, p. 132]{Ei}). 
By Proposition \ref{prop: reg2}, this implies  
that  $F(x) \bmod \mathfrak P$ is actually both $k$- and $\ell$-automatic.  
By Cobham's theorem, we obtain that the sequence of coefficients of $F(x)\bmod\mathfrak P$ is eventually 
periodic and hence $F(x) \bmod \mathfrak P$ is a rational function. 

\medskip

Note that since $\mathcal{S}$ is infinite, the intersection of all ideals in   
$\mathcal{S}$ is the zero ideal (see \cite[Lemma 4.16, p. 130]{Ei}). 
Moreover, $F(x)\bmod \mathfrak P$ is rational for every prime ideal $\mathfrak P\in\mathcal S$.    
Applying Lemma \ref{lem: rational}, we obtain that $F(x)$ is a rational function. This ends the proof. 
\end{proof}

\bigskip

\noindent{\bf \itshape Acknowledgement.}\,\,--- The authors would like to thank Michel Mend\`es France for his 
comments and encouragements. The first author is indebted to  \'Eric Delaygue for his help with Maple. 
He is also most grateful to Macha and Vadim for inspiring discussions all along the preparation of this paper. 

 \end{document}